\documentclass[dvipdfmx, reqno]{amsart}
\usepackage{amsthm,amsmath,amssymb,amsthm, bbm}
\usepackage[left=3.5cm, right=3.5cm, top=3cm, bottom=3cm, includefoot]{geometry}
\usepackage[all]{xy}
\usepackage{color}
\usepackage{appendix}
\usepackage{tikz}
\usepackage{tikz-cd}

\newtheoremstyle{mystyle}
  {}
  {}
  {}
  {}
  {\bfseries}
  {.}
  { }
  {}

\theoremstyle{mystyle}
\theoremstyle{definition}
\newtheorem{theorem}{Theorem}[section]
\newtheorem{proposition}{Proposition}[section]
\newtheorem{lemma}{Lemma}[section]
\newtheorem{corollary}{Corollary}[section]

\newtheorem{example}{Example}[section]
\newtheorem{definition}{Definition}[section]
\newtheorem{remark}{Remark}[section]
\numberwithin{equation}{section}

\begin{document}
\title[GICAR algebras and dynamics on DPP]{GICAR algebras and dynamics on determinantal point processes: discrete orthogonal polynomial ensemble case}
\author[R. Sato]{Ryosuke SATO}
\address{Department of Physics, Factualy of Science and Engineering, Chuo Univsersity, Kasuga, Bunkyo-ku, Tokyo 112-8551, Japan}
\email{r.sato@phys.chuo-u.ac.jp}

\begin{abstract}
    It is known that determinantal point processes have an intimate relation to operator algebras. In the paper, we extend this relationship to encompass dynamical aspects. Especially, we delve into two types of determinantal point processes. One is related to discrete orthogonal polynomials of hypergeometric type, and another is $z$-measures, which arise in the asymptotic representation theory of the symmetric groups. Within the framework of operator algebras, we investigate both unitary and stochastic dynamics on these point processes.
\end{abstract}

\maketitle

\allowdisplaybreaks{
\section{Introduction}
A point process on a (locally compact) space $\mathfrak{X}$ is a probability measure on the space $\mathcal{C}(\mathfrak{X})$ of (locally finite) point configurations in $\mathfrak{X}$. In the study of point processes, correlation functions play a role similar to that of moments in the study of random variables. Specially, they serve to characterize the statistical properties of point processes (see e.g., \cite{Lenard73}). If correlation functions are represented as determinants of a kernel defined on $\mathfrak{X}\times \mathfrak{X}$, then the point process is said to be \emph{determinantal}. Determinantal point processes arise in diverse branches of mathematics and physics, including random matrix theory, interacting particle systems, combinatorics, asymptotic representation theory, and so on. For surveys on these topics, see \cite{Soshnikov00, Lyons03, Johansson05, HKPV06, Borodin10, BG16} etc.

It is known that determinantal point processes have an intimate relation to operator algebras (see \cite{Lytvynov02,LM07, Olshanski20,Koshida21}). For arbitrary Hilbert space $\mathcal{H}$, a universal $C^*$-algebra with generators satisfying the canonical anti-commutation relation (CAR in short) over $\mathcal{H}$ is called the \emph{CAR algebra}, denoted by $\mathfrak{A}(\mathcal{H})$. The CAR algebra $\mathfrak{A}(\mathcal{H})$ admits a so-called gauge action by the torus group $\mathbb{T}$, and its fixed point subalgebra $\mathfrak{I}(\mathcal{H})$ is called the \emph{gauge invariant} (GI)CAR algebra. If $\mathcal{H}=\ell^2(\mathfrak{X})$ for a countable space $\mathfrak{X}$, then there exists a commutative $C^*$-subalgebra of $\mathfrak{I}(\mathcal{H})$ that is isomorphic to the $C^*$-algebra $C(\mathcal{C}(\mathfrak{X}))$ of continuous functions on $\mathcal{C}(\mathfrak{X})$. Hence, every state of $\mathfrak{I}(\mathcal{H})$ (i.e., positive normalized linear functional on $\mathfrak{I}(\mathcal{H})$) produces a point process on $\mathfrak{X}$. Particularly, so-called \emph{quasi-free} states of $\mathfrak{I}(\mathcal{H})$ induce determinantal point processes on $\mathfrak{X}$ (see Section \ref{sec:DPPGICAR}). Namely, there exists a correspondence from quasi-free states of $\mathfrak{I}(\mathcal{H})$ to determinantal point processes on $\mathfrak{X}$. However, there is no inverse correspondence. Olshanski \cite{Olshanski20} introduced a class of determinantal point processes, called \emph{perfect} measures, that are canonically associated with quasi-free states. In the paper \cite{Olshanski20}, a canonical representation of a GICAR algebra on the $L^2$-space with respect to a point process is discussed, and it is asked whether a vector state given by the constant function $\mathbbm{1}$ is quasi-free. On the other hand, by the so-called GNS (Gelfand--Naimark--Segal) construction, a quasi-free state of $\mathfrak{I}(\mathcal{H})$ provides a representation of $\mathfrak{I}(\mathcal{H})$ on a Hilbert space, and the quasi-free state can be described as a vector state given by a unit vector in the representation space. Hence, the $L^2$-space with respect to a perfect measure admits the GNS representation for the associated quasi-free state. In the paper, we also explore an isomorphism between the $L^2$-space with respect to a determinantal point process and the GNS representation space for the associated quasi-free state. Based on the isomorphism, we will apply an operator algebraic framework to study determinantal point processes.

The purpose of the paper is to extend the interaction between determinantal point processes and quasi-free states of GICAR algebras to encompass dynamical aspects. In several literature, many researchers investigate stochastic dynamics on determinantal point processes, i.e., stochastic dynamics on a configuration space with invariant measures given as determinantal point processes. For stochastic dynamics related to interacting particle systems (on discrete spaces), see \cite{SY02, CY09, LO13} etc. Moreover, for stochastic dynamics with representation-theoretic origin, we refer the reader to \cite{BO06b, Olshanski11, BO12, BO13, Petrov13} etc. In the paper, we will focus on a determinantal point process with the following properties:
\begin{itemize}
    \item A correlation kernel is given as a spectral projection of a self-adjoint operator.
    \item The $L^2$-space with respect to a determinantal point process is isomorphic to the GNS-representation space of the associated quasi-free state of a GICAR algebra.
\end{itemize}
We remark that the first property implies that a determinantal point process is provided from a quasi-free state of a GICAR algebra. Then, we can construct dynamics on the GICAR algebra by the self-adjoint operator in the first property, and dynamics on the $L^2$-space with respect to a determinantal point process can be constructed using the second property.

The organization of the paper is the following: In Section \ref{sec:DPPGICAR}, we summarize the basic facts of determinantal point processes on discrete spaces. We also introduce GICAR algebras and their quasi-free states. The relationship between determinantal point processes and GICAR algebra is also discussed. In Section \ref{sec:hyper_dop}, we present two types of determinantal point processes. One is related to discrete orthogonal polynomials of hypergeometric type, and another is $z$-measures, which arise in the asymptotic representation theory of the symmetric groups. One of their features is that their correlation kernels are given as spectral projections of self-adjoint difference operators. Moreover, $L^2$-spaces of these point processes are isomorphic to GNS-representation spaces of corresponding quasi-free states. These properties are used in Sections \ref{sec:unitary_dynamics} and \ref{sec:stochastic_dynamics} to construct dynamics on GICAR algebras, and they give rise to dynamics on determinantal point processes. See Propositions \ref{prop:second_quantization}, \ref{prop:constract_semigroup}. Based on the operator algebraic framework, we can describe the infinitesimal generators of these dynamics by creation and annihilation operators in CAR algebras. See Theorem \ref{thm:hamiltonian} and Corollaries \ref{cor:orth_poly_cre_ann}, \ref{cor:hyp_cre_ann}, \ref{cor:love}.

\section{Discrete point processes and GICAR algebras}\label{sec:DPPGICAR}
\subsection{Discrete point processes}
In the paper, we deal with point processes on only discrete spaces. Let $\mathfrak{X}$ be a countable set with the discrete topology, which is referred as state space. The set of (\emph{simple}) \emph{point configurations} $\mathcal{C}(\mathfrak{X})$ on $\mathfrak{X}$ is defined as $\mathcal{C}(\mathfrak{X}):=\{0, 1\}^\mathfrak{X}$ equipped with the product topology. Every $\omega=(\omega_x)_{x\in\mathfrak{X}}\in\mathcal{C}(\mathfrak{X})$ naturally represents a state that $x\in\mathfrak{X}$ is occupied by a particle if $\omega_x=1$; otherwise, $x$ is empty. A Borel probability measure on $\mathcal{C}(\mathfrak{X})$ is called a \emph{point process} on $\mathfrak{X}$. If the support of a point process is in $\mathcal{C}_N(\mathfrak{X}):=\{\omega\in\mathcal{C}(\mathfrak{X})\mid \sum_{x\in\mathfrak{X}}\omega_x=N\}$, then we call it a \emph{$N$-point process}. Let $M$ be a point process on $\mathfrak{X}$. For every $n\geq1$ we define the \emph{$n$-th correlation function} $\rho_M^{(n)}\colon\mathfrak{X}^n\to \mathbb{R}$ by
\[\rho_M^{(n)}(x_1, \dots, x_n):=\begin{cases}M(\{\omega\in\mathcal{C}(\mathfrak{X})\mid \omega_{x_i}=1\text{ for }i=1, \dots, n\})&x_i\text{'s are pairwise distinct},\\0&\text{otherwise.}\end{cases}\]
If there exits a function $K$ on $\mathfrak{X}\times \mathfrak{X}$ such that
\[\rho_M^{(n)}(x_1, \dots, x_n)=\det[K(x_i, x_j)]_{i, j=1}^n\]
holds for every $n\geq1$, then $M$ is said to be \emph{determinantal}, and $K$ is called a \emph{correlation kernel}. We remark that a correlation kernel is \emph{not} unique even if it exists.

We introduce an important example of determinantal point processes that relate to orthogonal polynomials. Let us assume that $\mathfrak{X}\subset \mathbb{R}$ and a function $w\colon \mathfrak{X}\to \mathbb{R}_{>0}$ has finite moments of orders $0, 1, \dots, 2N-2$ for some $N\geq1$. Namely, $\sum_{x\in\mathfrak{X}}x^nw(x)<\infty$ holds for $n=0, \dots, 2N-2$. By the assumption, we have
\begin{equation}\label{eq:normalized_constant}
    0<Z_{w, N}:=\sum_{\omega\in\mathcal{C}_N(\mathfrak{X})}V_N^2(\omega)\prod_{x\in\mathfrak{X}:\omega_x=1}w(x)<\infty,
\end{equation}
where $V_N^2(\omega):=\prod_{1\leq i<j\leq N}(x_i-x_j)^2$ and $x_1, \dots, x_N$ are $N$-points in $\mathfrak{X}$ satisfying $\omega_{x_i}=1$. We remark that $V_N(x_1, \dots, x_N):=\prod_{1\leq i<j \leq N}(x_i-x_j)$ is anti-symmetric in the order of $x_1, \dots, x_N$, but $V_N(x_1, \dots, x_N)^2$ is independent of their order. Hence $V_N^2(\omega)$ is well defined as a function on $\mathcal{C}_N(\mathfrak{X})$. The \emph{discrete orthogonal polynomial ensemble} $M_{w, N}$ (with weight $w$) is defined by
\[M_{w, N}(\omega):=\frac{1}{Z_{w, N}}V_N^2(\omega)\prod_{x\in\mathfrak{X}: \omega_x=1}w(x) \quad (\omega\in\mathcal{C}_N(\mathfrak{X})).\]
It is well known that $M_{w, N}$ is determinantal, and its correlation kernel is expressed by discrete orthogonal polynomials with respect to $w$. By the assumption on $w$, the functions $1, x, \dots, x^{N-1}$ on $\mathfrak{X}$ belong to $L^2(\mathfrak{X}, w)$. Hence, by orthogonalizing them, we obtain monic orthogonal polynomials $\tilde p_0(x)\equiv1, \tilde p_1(x), \dots, \tilde p_{N-1}(x)$. Moreover, let $p_n(x):=\tilde p_n(x) w(x)^{1/2}/\|\tilde p_n\|_{L^2(\mathfrak{X}, w)}$. Then $p_0(x), \dots, p_n(x)$ are orthonormal in $\ell^2(\mathfrak{X})$. We define a kernel $K^{CD}_{w, N}$ on $\mathfrak{X}\times \mathfrak{X}$ by
\begin{equation}\label{eq:CD}
    K^{CD}_{w, N}(x, y):=\sum_{n=0}^{N-1}p_n(x)p_n(y) \quad (x, y\in\mathfrak{X}),
\end{equation}
called the \emph{normalized Christoffel--Darboux kernel} with weight $w$. It is known that $K^{CD}_{w, N}$ is a correlation kernel of $M_{w, N}$ (see e.g., \cite[Lemma 2.8]{Konig}). We remark that the associated integral operator on $\ell^2(\mathfrak{X})$ is nothing but the orthogonal projection onto the linear span of $p_0(x), \dots, p_{N-1}(x)$.

\subsection{GICAR algebras and quasi-free states}\label{subsec:GICAR}
Here we provide a brief summary about CAR algebras, GICAR algebras, and their quasi-free states. See e.g. \cite{BratteliRobinson2, EK:book} for more details.

Let $\mathcal{H}$ be a separable complex Hilbert space. The \emph{CAR algebra } $\mathfrak{A}(\mathcal{H})$ is the universal unital $C^*$-algebra generated by $a(h)$ ($h\in \mathcal{H}$) such that the mapping $h\in \mathcal{H}\mapsto a^*(h):=a(h)^*\in\mathfrak{A}(\mathcal{H})$ is linear and the canonical anticommutation relation (CAR) holds true, that is,
\[a(h)a(k)+a(k)a(h)=0, \quad a^*(h)a(k)+a(k)a^*(h)=\langle h, k\rangle 1\quad (h, k\in \mathcal{H}),\]
where the inner product is linear in the first argument. By the universality of $\mathfrak{A}(\mathcal{H})$, a unitary operator $U\colon \mathcal{H}\to \mathcal{H}$ provides the $*$-automorphism $\gamma_U$ on $\mathfrak{A}(\mathcal{H})$, called the \emph{Bogoliubov transform}, by $\gamma_U(a(h))=a(Uh)$ for every $h\in\mathcal{H}$. In particular, the so-called \emph{gauge action} $\gamma\colon \mathbb{T}\curvearrowright\mathfrak{A}(\mathcal{H})$ is defined as $\gamma_\lambda:=\gamma_{\lambda1}$ for any $\lambda\in \mathbb{T}$. The fixed point subalgebra
\[\mathfrak{I}(\mathcal{H}):=\{x\in\mathfrak{A}(\mathcal{H})\mid \gamma_\lambda(x)=x\text{ for all }\lambda\in\mathbb{T}\}\]
is called the \emph{gauge invariant} CAR algebra (the GICAR algebra in short). We remark that $\mathfrak{I}(\mathcal{H})$ is a $C^*$-subalgebra of $\mathfrak{A}(\mathcal{H})$ and generated by elements of the form $a^*(h)a(k)$ for any $h, k\in\mathcal{H}$.

The most fundamental representation of $\mathfrak{A}(\mathcal{H})$ is defined on the \emph{anti-symmetric Fock space} $\mathcal{F}_a(\mathcal{H}):=\bigoplus_{n=0}^\infty\bigwedge^n\mathcal{H}$, where $\bigwedge^n\mathcal{H}$ is the $n$-th anti-symmetric tensor product of $\mathcal{H}$ and $\bigwedge^0\mathcal{H}$ is the linear span of a distinguished vector $\Omega$, called a \emph{vacuum} vector. The \emph{Fock representation} $(\pi, \mathcal{F}_a(\mathcal{H}))$ of $\mathfrak{A}(\mathcal{H})$ is defined by
\[\pi(a^*(h))\Omega:=h, \quad \pi(a^*(h))h_1\wedge \cdots \wedge h_n:=h\wedge h_1\wedge \cdots \wedge h_n,\]
or equivalently
\[\pi(a(h))\Omega:=0, \quad \pi(a(h))h_1\wedge \cdots \wedge h_n:=\sum_{i=1}^n(-1)^{i-1}\langle h_i, h\rangle h_1\wedge \cdots \check{h_i}\cdots \wedge h_n,\]
where $\check{h_i}$ means that $h_i$ is omitted. We remark that $\mathfrak{A}(\mathcal{H})$ and $\pi(\mathfrak{A}(\mathcal{H}))$ are $*$-isomorphic. In what follows, $\mathfrak{A}(\mathcal{H})$ is identified with $\pi(\mathfrak{A}(\mathcal{H}))$, and we use the same symbol $a(h)$ and $a^*(h)$ to denote $\pi(a(h))$ and $\pi(a^*(h))$, respectively. The vector state $\varphi_0$ on $\mathfrak{A}(\mathcal{H})$ given by $\Omega$ is called the \emph{Fock state}, that is, the Fock state $\varphi_0$ on $\mathfrak{A}(\mathcal{H})$ is defined by $\varphi_0(x)=\langle x\Omega, \Omega\rangle$.

A state $\varphi$ on $\mathfrak{A}(\mathcal{H})$ is said to be \emph{quasi-free} if we have
\[\varphi(a^*(h_n)\cdots a^*(h_1)a(k_1)\cdots a(k_m))=\delta_{n, m}\det[\varphi(a^*(h_i)a(k_j))]_{i, j=1}^n\]
for every $h_1, \dots, h_n, k_1, \dots, k_m\in\mathcal{H}$ and $n, m\geq1$. By the Riesz representation theorem, there is a unique bounded linear operator $K$ on $\mathcal{H}$ such that $0\leq K\leq 1$ and $\varphi(a^*(h)a(k))=\langle Kh, k\rangle$ for every $h, k\in\mathcal{H}$. Conversely, for any positive contraction operator $K$ on $\mathcal{H}$ there is the associated quasi-free state, denoted by $\varphi_K$. In particular, if $K=0$, then the associated quasi-free state is nothing but the Fock state $\varphi_0$. We remark that a quasi-free state $\varphi_K$ is gauge invariant, that is, $\varphi_K\circ\gamma_\lambda=\varphi_K$ for any $\lambda\in\mathbb{T}$. Hence, we obtain a unique unitary representation $(\Gamma, \mathcal{F}_a(\mathcal{H}))$ of $\mathbb{T}$ such that for any $\lambda\in\mathbb{T}$ and $x\in\mathfrak{A}(\mathcal{H})$
\[\Gamma_\lambda\Omega=\Omega, \quad \Gamma_\lambda x = \gamma_\lambda(x)\Gamma_\lambda.\]
See e.g., \cite[Corollary 2.3.17]{BratteliRobinson1}. Here we recall an explicit realization of quasi-free states for given positive contraction operator $K$ on $\mathcal{H}$. Let $S:=(1-K)^{1/2}$ and $T:=K^{1/2}$. By the universality of $\mathfrak{A}(\mathcal{H})$, we obtain a unital $*$-homomorphism $\rho_K\colon \mathfrak{A}(\mathcal{H})\to B(\mathcal{F}_a(\mathcal{H})\otimes \mathcal{F}_a(\overline{\mathcal{H}}))$ by
\[\rho_K(a(h)):=a(Sh)\otimes \bar{\Gamma}+1\otimes a^*(\overline{T h}) \quad (h\in\mathcal{H}),\]
where $\Gamma:=\Gamma_{-1}$ and $\bar\Gamma$ is its conjugate operator. For any $h_1, \dots, h_n, k_1, \dots, k_m\in\mathcal{H}$ we have
\begin{align*}
     & \langle \rho_K(a^*(h_n)\cdots a^*(h_1)a(k_1)\cdots a(k_m))\Omega\otimes \bar\Omega, \Omega\otimes\bar\Omega\rangle                                 \\
     & =\langle \Omega\otimes \overline{Tk_1}\wedge \cdots \wedge\overline{Tk_n}, \Omega\otimes \overline{Th_1}\wedge \cdots \wedge\overline{Th_n}\rangle \\
     & =\delta_{n, m}\det[\langle \overline{Tk_j}, \overline{Th_i}\rangle]_{i, j=1}^n                                                                     \\
     & =\delta_{n, m}\det[\langle Kh_i, k_j\rangle]_{i, j=1}^n                                                                                            \\
     & =\varphi_K(a^*(h_n)\cdots a^*(h_1)a(k_1)\cdots a(k_m)).
\end{align*}
Thus, the GNS representation of $\varphi_K$ is given by (a subrepresentation of) $(\rho_K, \mathcal{F}_a(\mathcal{H})\otimes \mathcal{F}_a(\overline{\mathcal{H}}))$. The space of the GNS representation associated with $\varphi_K|_{\mathfrak{I}(\mathcal{H})}$ is given as follows:

\begin{lemma}\label{lem:GNS_space}
    The closed linear span of $\rho_K(\mathfrak{I}(\mathcal{H}))\Omega\otimes\bar\Omega$ coincides with $\mathcal{H}_K$ defined as
    \[\mathcal{H}_K:=\bigoplus_{n=0}^\infty\bigwedge^n\overline{\mathrm{Ran}}(S)\otimes \bigwedge^n\overline{\overline{\mathrm{Ran}}(T)},\]
    where $\overline{\mathrm{Ran}}(S)$ and $\overline{\mathrm{Ran}}(T)$ are the closures of the ranges of $S$ and $T$, respectively.
\end{lemma}
\begin{proof}
    Since $\mathcal{H}_K$ is invariant under the $\rho_K(\mathfrak{I}(\mathcal{H}))$, we have $\rho_K(\mathfrak{I}(\mathcal{H}))\Omega\otimes\bar\Omega\subset \mathcal{H}_K$. Conversely, by induction, we can show that for any $n\geq1$ and $h_1, \dots, h_n, k_1, \dots, k_n\in \mathcal{H}$ there exists a $X_K(h_1, \dots, h_n; k_1, \dots, k_n)\in\mathfrak{I}(\mathcal{H})$ such that
    \[\rho_K(X_K(h_1, \dots, h_n;k_1, \dots, k_n))\Omega\otimes\bar\Omega=Sh_1\wedge\cdots\wedge Sh_n\otimes \overline{Tk_1}\wedge \cdots \wedge \overline{Tk_n}.\]
    Hence $\mathcal{H}_K$ is contained in $\overline{\rho_K(\mathfrak{I}(\mathcal{H}))\Omega\otimes\bar\Omega}$. Namely, they coincide with each other.
\end{proof}

\begin{remark}
    By Lemma \ref{lem:GNS_space}, $(\rho_K, \mathcal{H}_K, \Omega\otimes\overline{\Omega})$ is the GNS triple associated with $\varphi_K$. Namely, $\Omega\otimes\overline{\Omega}$ is a cyclic vector for $\rho_K(\mathfrak{I}(\mathcal{H}))$, and $\langle \rho_K(x)\Omega\otimes \overline{\Omega}, \Omega\otimes \overline{\Omega}\rangle=\varphi_K(x)$ holds for any $x\in\mathfrak{I}(\mathcal{H})$.
\end{remark}

\subsection{Diagonalization of GICAR algebras}
GICAR algebras (and also CAR algebras) belongs to the class of AF (approximately finite dimensional)-algebras. Here we briefly summarize about the diagonalization of AF-algebras in the sense of \cite[Chapter 1]{SV:book}, specifically focusing on GICAR algebras. From this perspective, (quasi-free) states of GICAR algebras can be related to (determinantal) point processes.

Let us fix $\{e_x\}_{x\in\mathfrak{X}}$ an orthonormal basis for $\mathcal{H}$, where $\mathfrak{X}$ is a countable index set. For any finite subset $\mathfrak{F}\Subset\mathfrak{X}$ we denote by $\mathcal{H}_\mathfrak{F}$ the linear span of $\{e_x\}_{x\in\mathfrak{F}}$ and define $\mathfrak{A}_\mathfrak{F}:=\mathfrak{A}(\mathcal{H}_\mathfrak{F})$ and $\mathfrak{I}_\mathfrak{F}:=\mathfrak{A}_\mathfrak{F}\cap \mathfrak{I}(\mathcal{H})$. Under the Fock representation, the gauge action $\gamma^\mathfrak{F}\colon\mathbb{T}\curvearrowright\mathfrak{A}_\mathfrak{F}$ is given as $\mathrm{Ad}\Gamma^\mathfrak{F}_\lambda$, where $\Gamma^\mathfrak{F}_\lambda$ ($\lambda\in\mathbb{T}$) is defined by
\[\Gamma^\mathfrak{F}_\lambda\Omega=\Omega, \quad \Gamma^\mathfrak{F}_\lambda h_1\wedge \cdots \wedge h_n:=(\lambda h_1)\wedge\cdots \wedge (\lambda h_n) \quad (h_1,\dots, h_n\in\mathcal{H}_\mathfrak{F}).\]
Hence we have $\mathfrak{I}_\mathfrak{F}=\Gamma^\mathfrak{F}(\mathbb{T})'=\bigoplus_{n=0}^{|\mathfrak{F}|} B(\bigwedge^n\mathcal{H}_\mathfrak{F})$.

For any $\underline{x}=\{x_1,\dots, x_k\}\subseteq \mathfrak{F}$ let $p_{\underline{x}}\in\mathfrak{A}_\mathfrak{F}$ be the projection operator onto $\mathbb{C}e_{x_1}\wedge \cdots \wedge e_{x_k}$. We define $C_\mathfrak{F}:=\bigoplus_{\underline{x}\subseteq \mathfrak{F}}\mathbb{C}p_{\underline{x}}$, which is a maximal abelian subalgebra of $\mathfrak{I}_\mathfrak{F}$ and $C_\mathfrak{F}\cong C(\mathcal{C}(\mathfrak{F}))$. For any $\underline{x}=\{x_1, \dots, x_k\}$ we define
\[q_{\underline{x}}:=a_{x_1}^*a_{x_1}\cdots a_{x_k}^*a_{x_k}=a_{x_k}^*\cdots a_{x_1}^*a_{x_1}\cdots a_{x_k},\]
where $a_{x_i}:=a(e_{x_i})$. Then $q_{\underline{x}}$ belongs to $C_\mathfrak{F}$ and associates with the characteristic function of
\[\{\omega=(\omega_x)_{x\in\mathfrak{X}}\in\mathcal{C}(\mathfrak{F})\mid \omega_{x_i}=1 \text{ for }i=1, \dots, n\}.\]
Thus, $C_\mathfrak{F}$ is generated by the $q_{\underline{x}}$ for any $\underline{x}\subseteq \mathfrak{F}$.

If $\mathfrak{F}\subseteq \mathfrak{G}\Subset\mathfrak{X}$, then $\mathcal{H}_\mathfrak{F}\subseteq \mathcal{H}_\mathfrak{G}$ and $\mathfrak{I}_\mathfrak{F}$ is identified with a subalgebra of $\mathfrak{I}_\mathfrak{G}$. Moreover, $\mathfrak{I}(\mathcal{H})$ is generated by $\bigcup_{\mathfrak{F}\Subset\mathfrak{X}}\mathfrak{I}_\mathfrak{F}$, and hence $\mathfrak{I}(\mathcal{H})$ is an AF-algebra. Since $C_\mathfrak{F}$ is generated by the $q_{\underline{x}}$ for any $\underline{x}\subseteq \mathfrak{F}$, we have $C_\mathfrak{F}\subseteq C_\mathfrak{G}$. Under the identifications $C_\mathfrak{F}\cong C(\mathcal{C}(\mathfrak{F}))$ and $C_\mathfrak{G}\cong C(\mathcal{C}(\mathfrak{G}))$, the embedding $C_\mathfrak{F}\subseteq C_\mathfrak{G}$ is the dual of the mapping $(\omega_x)_{x\in\mathfrak{G}}\in\mathcal{C}(\mathfrak{G})\mapsto (\omega_x)_{x\in\mathfrak{F}}\in \mathcal{C}(\mathfrak{F})$. Therefore,
\[\varinjlim C_\mathfrak{F}\cong \varinjlim C(\mathcal{C}(\mathfrak{F}))\cong C(\varprojlim\mathcal{C}(\mathfrak{F}))\]
holds. Here $\varprojlim\mathcal{C}(\mathfrak{F})$ is the projective limit of the $\mathcal{C}(\mathfrak{F})$ with respect to the above maps, and it is homeomorphic to $\mathcal{C}(\mathfrak{X})$. We remark that the inductive limits and projective limit are taken in the categories of (unital commutative) $C^*$-algebras and topological spaces, respectively.

By the observation, $C(\mathcal{C}(\mathfrak{X}))$ is $*$-isomorphic to $\varinjlim C_\mathfrak{F}\subset\mathfrak{I}(\mathcal{H})$. Thus, by the Riesz--Markov--Kakutani theorem, for any state $\varphi$ on $\mathfrak{I}(\mathcal{H})$ we obtain a Borel probability measure $M_\varphi$ on $\mathcal{C}(\mathfrak{X})$, that is, a point process $M_\varphi$ on $\mathfrak{X}$, such that
\[\rho^{(n)}(x_1, \dots, x_n)=\varphi(q_{\{x_1, \dots, x_n\}})=\varphi(a^*_{x_n}\cdots a^*_{x_1}a_{x_1}\cdots a_{x_n})\]
holds for any $n\geq1$ and $x_1, \dots, x_n\in\mathfrak{X}$. In particular, the point process $M_{\varphi_K}$ associated with a quasi-free state $\varphi_K$ is determinantal, and its correlation kernel is given by $K(x, y):=\langle Ke_x, e_y\rangle$.

\begin{remark}
    Any quasi-free states on $\mathfrak{I}(\mathcal{H})$ (or equivalently, any positive contraction operator on $\mathcal{H}$) provides determinantal point processes, but this correspondence is not one-to-one. More precisely, there is no inverse correspondence from determinantal point processes to quasi-free states in general.
\end{remark}

\section{Determinantal point processes and self-adjoint difference operators}\label{sec:hyper_dop}
In this section, we discuss discrete determinantal point processes, where correlation kernels are derived as spectral projections of self-adjoint difference operators. Olshanski extensively examined such determinantal point processes in \cite{Olshanski08}. In addition to his work, we present several facts from the general theory of discrete orthogonal polynomials of hypergeometric type.

First, we recall necessary facts on discrete orthogonal polynomials of hypergeometric type. See \cite[Chapter 2]{NSU} for more details. Throughout this section, a state space $\mathfrak{X}$ is assumed to be $\mathbb{Z}+a$ or $\mathbb{Z}_{\geq 0}+ a$ for some $a\in \mathbb{R}$. We define two difference operators $\Delta$ and $\nabla$ on $\mathfrak{X}$ by
\[\Delta y(x):=y(x+1)-y(x), \quad \nabla y(x):=y(x)-y(x-1)\]
for any function $y$ on $\mathfrak{X}$, where $y(x-1):=0$ when $x- 1\not\in\mathfrak{X}$. Let $\sigma(x)$ and $\tau(x)$ be polynomials with $\deg \sigma(x)\leq 2$ and $\deg \tau(x)\leq1$. We consider the following difference equation
\begin{equation}\label{eq:hypergeometric}
    \sigma(x)\Delta\nabla y(x)+\tau(x)\Delta y(x)+\lambda y(x)=0
\end{equation}
of hypergeometric type, where $\lambda$ is a constant. Let $m_n:=\tau'n+\sigma''n(n-1)/2$ for any $n\geq0$. If $\lambda = -m_0=0$, then a constant function is a solution of Equation \eqref{eq:hypergeometric}. Moreover, if $m_n\neq m_k$ for $k=0, \dots, n-1$, then there exists a polynomial $y_n(x)$ of degree $n$ that is a solution of Equation \eqref{eq:hypergeometric} with $\lambda=-m_n$.

We additionally assume that $w\colon\mathfrak{X}\to \mathbb{R}_{>0}$ satisfies that $\Delta[\sigma(x)w(x)]=\tau(x)w(x)$ and the following boundary condition:
\begin{itemize}
    \item $x^n\sigma(x)w(x)\to 0$ as $x\to \pm \infty$ for $n=0, 1, \dots, 2N-1$ when $\mathfrak{X}=\mathbb{Z}+a$,
    \item $a^n\sigma(a)w(a)=0$ and $x^nw(x)\to 0$ as $x\to \infty$ for $n=0, 1, \dots, 2N-1$ when $\mathfrak{X}=\mathbb{Z}_{\geq 0}+a$.
\end{itemize}
Then $(y_n)_{n=0}^{N-1}$ lie in $L^2(\mathfrak{X}, w)$ and they are orthogonal (see \cite[Section 2.3]{NSU}). In this case, we call the $y_n(x)$ discrete orthogonal polynomials \emph{of hyepergeometic type}.

Let $\sigma$, $\tau$, and $w$ be the same as above, and $\sigma$ is assumed to be that $\mu_x:=\sigma(x)>0$ on $\mathfrak{X}$. Then $\lambda_x:=\sigma(x)+\tau(x)=\sigma(x+1)w(x+1)/w(x)>0$ holds for any $x\in \mathfrak{X}$. We define a operator $\mathcal{D}$ on $L^2(\mathfrak{X}, w)$ by
\begin{align*}
    [\mathcal{D}f](x)
     & :=\sigma(x)\Delta\nabla f(x)+\tau(x)\Delta f(x)    \\
     & =\lambda_xf(x+1)-(\mu_x+\lambda_x)f(x)+\mu_xf(x-1)
\end{align*}
with domain $\mathrm{dom}(\mathcal{D}):=\{f\in L^2(\mathfrak{X}, w)\mid \mathcal{D}y\in L^2(\mathfrak{X}, w)\}$.

\begin{remark}\label{rem:birth_death_process}
    The operator $\mathcal{D}$ has the form of infinitesimal generators of (\emph{bilateral}) \emph{birth death processes} on $\mathfrak{X}$ (see e.g., \cite{Liggett:book2010}). The connection between birth death processes and orthogonal polynomials was investigated in \cite{Schoutens:book}, for instance.
\end{remark}

We also consider $D:=\sqrt{w}\mathcal{D}\sqrt{w}^{-1}$ on $\ell^2(\mathfrak{X})$ with $\mathrm{dom}(D):=\{f\in \ell^2(\mathfrak{X})\mid Df\in\ell^2(\mathfrak{X})\}$. Then we have
\begin{align}\label{eq:differene_operator}
    [Df](x)
     & =\lambda_x\sqrt{\frac{w(x)}{w(x+1)}}f(x+1)-(\mu_x+\lambda_x)f(x)+\mu_x\sqrt{\frac{w(x)}{w(x-1)}}f(x-1) \nonumber \\
     & =\sqrt{\mu_{x+1}\lambda_x}f(x+1)-(\mu_x+\lambda_x)f(x)+\sqrt{\mu_x\lambda_{x-1}}f(x-1)
\end{align}
since $\sigma(x+1)w(x+1)=(\sigma(x)+\tau(x))w(x)$.

Let $\{e_x\}_{x\in\mathfrak{X}}$ denote the canonical orthonormal basis for $\ell^2(\mathfrak{X})$ and $\ell^2_0(\mathfrak{X})$ the linear span of them. By definition, we have $\ell^2_0(\mathfrak{X})\subset \mathrm{dom}(D)$, and hence $D$ is densely defined.

The following two lemmas might be known to experts (see e.g., \cite{Kreer}), but we give the proofs for the reader's convenience.
\begin{lemma}
    The operator $D$ is closed symmetric.
\end{lemma}
\begin{proof}
    To show that $D$ is closed, we suppose that $f_n\in\mathrm{dom}(D)$ and $Df_n$ converge to $f$ and $g$ as $n\to \infty$, respectively. Thus, for any $x\in\mathfrak{X}$ we have
    \begin{align*}
        g(x)
         & =\lim_{n\to\infty}\sqrt{\mu_{x+1}\lambda_x}f_n(x+1)-(\mu_x+\lambda_x)f_n(x)+\sqrt{\mu_x\lambda_{x-1}}f_n(x-1) \\
         & =\sqrt{\mu_{x+1}\lambda_x}f(x+1)-(\mu_x+\lambda_x)f(x)+\sqrt{\mu_x\lambda_{x-1}}f(x-1)                        \\
         & =[Df](x).
    \end{align*}
    Namely, $g=Df$, and hence $D$ is closed. Next, we show that $D$ is symmetric. Let $D_0$ denote the restriction of $D$ to $\ell^2_0(\mathfrak{X})$. For any $l, k\in\mathfrak{X}$ with $l<k$, we define $[l, k]:=\{x\in\mathfrak{X}\mid l\leq x\leq k\}$. For any $f\in \mathrm{dom}(D)$ we have
    \begin{align*}
        \|D_0(f_{[l, k]})-Df\|^2
         & =\sum_{x\in\mathfrak{X}\backslash[l, k]}|Df(x)|^2+|\sqrt{\mu_l\lambda_{l-1}}f(l-1)|^2+|\sqrt{\mu_{k+1}\lambda_k}f(k)|^2
        \to 0
    \end{align*}
    as $l\to -\infty$ and $k\to \infty$. Thus, the graph of $D$ is contained in the closure of the graph of $D_0$, i.e., $D=\overline{D_0}$, and hence $D_0^*=(\overline{D_0})^*=D^*$. On the other hand, $\langle D_0f, g\rangle=\langle f, D_0g\rangle$ holds for any $f, g\in \ell^2_0(\mathfrak{X})$. Thus, we have $D_0\subset D_0^*=D^*$ and $D\subset D^*$. Namely, $D$ is symmetric.
\end{proof}

\begin{lemma}\label{lem:negative_semidefinite}
    The operator $D$ is negative semidefinite, that is, $\langle Df, f\rangle\leq 0$ for any $f\in\mathrm{dom}(D)$
\end{lemma}
\begin{proof}
    For any $f\in \ell^2_0(\mathfrak{X})$ there exist $l, k\in\mathfrak{X}$ such that $l<k$ and the support of $f$ is contained in $[l, k]$. Then we have
    \begin{align*}
        \langle D_0f, f\rangle
        =\sum_{x=l}^k [D_0f](x)\overline{f(x)}
        =-\sum_{x=l}^k|\sqrt{\mu_{x+1}}f(x+1)-\sqrt{\lambda_x}f(x)|^2-\mu_l|f(l)|^2
        \leq 0.
    \end{align*}
    Namely, $D_0$ is negative semidefinite, and so is $D=\overline{D_0}$.
\end{proof}

The following statement was shown in \cite{Kreer} under a certain condition about $\mu_x$ and $\lambda_x$ ($x\in\mathfrak{X}$). Although this condition does not holds in our case (the case of hypergeometric type), we can show the following, by slightly modifying the proof of \cite[Lemma 1]{Kreer}:
\begin{proposition}\label{prop:compact_resolvent}
    The operator $D$ has a compact resolvent at some $\zeta>0$, i.e., $(D-\zeta)^{-1}$ is a compact operator. Moreover, $D$ is self-adjoint and there exists a orthonormal basis $(f_n)_{n=0}^\infty$ for $\ell^2(\mathfrak{X})$ such that $Df_n=m_nf_n$ for any $n\geq0$, where $0\geq m_0\geq m_1\geq \cdots$ and $m_n\to -\infty$ as $n\to \infty$.
\end{proposition}
\begin{proof}
    The latter statement follows from the former (see e.g., \cite[Section X.1]{RS:book2} and \cite[Theorem XIII.64]{RS:book4}). Thus, it suffices to prove the first statement. For any $\zeta>0$ two operators $T_\zeta$ and $A$ on $\ell^2(\mathfrak{X})$ is defined by
    \[[T_\zeta f](x):=-(\mu_x+\lambda_x+t_x+\zeta)f(x), \quad [Af](x):=\sqrt{\mu_{x+1}\lambda_x}f(x+1)+t_xf(x)+\sqrt{\mu_x\lambda_{x-1}}f(x-1),\]
    where $t_x:=2\sqrt{\mu_x\lambda_x}/(\sqrt{6}-1)$. By definition, we have $D-\zeta=T_\zeta+A=(1+AT_\zeta^{-1})T_\zeta$ and
    \[[AT_\zeta^{-1}f](x)=-\frac{\sqrt{\mu_{x+1}\lambda_x}f(x+1)}{\mu_{x+1}+\lambda_{x+1}+t_{x+1}+\zeta}-\frac{t_xf(x)}{\mu_x+\lambda_x+t_x+\zeta}-\frac{\sqrt{\mu_x\lambda_{x-1}}f(x-1)}{\mu_{x-1}+\lambda_{x-1}+t_{x-1}+\zeta}.\]
    We show that for sufficiently large $\xi>0$ the above coefficients
    \[a_x:=\frac{\sqrt{\mu_{x+1}\lambda_x}}{\mu_{x+1}+\lambda_{x+1}+t_{x+1}+\zeta}, \quad b_x:=\frac{t_x}{\mu_x+\lambda_x+t_x+\zeta}, \quad c_x:=\frac{\sqrt{\mu_x\lambda_{x-1}}}{\mu_{x-1}+\lambda_{x-1}+t_{x-1}+\zeta}\]
    are less than $1/\sqrt{6}$. Then we have $\|AT_\zeta^{-1}\|<1$.
    First, for any $x\in\mathfrak{X}$ we have
    \[b_x\leq \frac{t_x}{2\sqrt{\mu_x\lambda_x}+t_x+\zeta}<\frac{1}{\sqrt{6}}.\]
    Next, we have
    \begin{align*}
        a_x\leq \frac{\sqrt{\mu_{x+1}\lambda_x}}{2\sqrt{\mu_{x+1}\lambda_{x+1}}+t_{x+1}+\zeta},
    \end{align*}
    and the right-hand side is less than $1/\sqrt{6}$ if and only if
    \[\sqrt{\mu_{x+1}\lambda_x}<\frac{1}{\sqrt{6}}(2\sqrt{\mu_{x+1}\lambda_{x+1}}+t_{x+1}+\zeta)=\frac{2\sqrt{\mu_{x+1}\lambda_{x+1}}}{\sqrt{6}-1}+\frac{\zeta}{\sqrt{6}}\]
    holds. Moreover, we have
    \[\sqrt{\mu_{x+1}\lambda_x}-\frac{2\sqrt{\mu_{x+1}\lambda_{x+1}}}{\sqrt{6}-1}=\sqrt{\mu_{x+1}\lambda_x}\left(1-\frac{2\sqrt{\lambda_{x+1}/\lambda_x}}{\sqrt{6}-1}\right)\to -\infty\]
    as $|x|\to \infty$. Thus, $a_x<1/\sqrt{6}$ holds for any $x\in\mathfrak{X}$ if $\zeta$ is sufficiently large. Similarly, we have
    \[c_x\leq \frac{\sqrt{\mu_x\lambda_{x-1}}}{2\sqrt{\mu_{x-1}\lambda_{x-1}}+t_{x-1}+\zeta}\]
    and the right-hand side is less than $1/\sqrt{6}$ if and only if
    \[\sqrt{\mu_x\lambda_{x-1}}<\frac{1}{\sqrt{6}}(2\sqrt{\mu_{x-1}\lambda_{x-1}}+t_{x-1}+\zeta)=\frac{2\sqrt{\mu_{x-1}\lambda_{x-1}}}{\sqrt{6}-1}+\frac{\zeta}{\sqrt{6}}\]
    holds. Moreover, we have
    \[\sqrt{\mu_x\lambda_{x-1}}-\frac{2\sqrt{\mu_{x-1}\lambda_{x-1}}}{\sqrt{6}-1}=\sqrt{\mu_x\lambda_{x-1}}\left(1-\frac{2\sqrt{\mu_{x-1}/\mu_x}}{\sqrt{6}-1}\right)\to -\infty\]
    as $|x|\to\infty$. Thus, $c_x<1/\sqrt{6}$ holds for any $x\in\mathfrak{X}$ if $\zeta$ is sufficiently large. Therefore, we can take $\zeta>0$ such that $\|AT_\zeta^{-1}\|<1$. Hence the Neumann series of $AT_\zeta^{-1}$ converges. Thus, $1+AT_\zeta^{-1}$ is invertible, and $(1+AT_\zeta^{-1})$ is bounded. Moreover, $(D-\zeta)^{-1}=T_\zeta^{-1}(1+AT_\zeta^{-1})^{-1}$ is a compact operator since $T_\zeta^{-1}$ is a compact operator.
\end{proof}

We give three examples of difference operators of hypergeometric type. All of them satisfy that $\sigma(x)>0$ on $\mathfrak{X}$. Thus, by the above results, the corresponding difference operators are self-adjoint.

\begin{example}[Meixner polynomials]\label{ex:Meixner}
    Let $\mathfrak{X}=\mathbb{Z}_{\geq0}$ and $w_{\beta, \xi}\colon \mathfrak{X}\to \mathbb{R}_{>0}$ be defined by
    \[w_{\beta, \xi}(x):=\frac{(\beta)_x \xi^x}{x!},\]
    where $\beta>0$, $\xi\in (0, 1)$, and $(\beta)_x:=\beta(\beta+1)\cdots (\beta+x-1)$. Namely, $w_{\beta, \xi}$ is the negative binomial distribution with parameters $\beta, \xi$. Then $\Delta[\sigma(x) w_{\beta, \xi}(x)]=\tau(x)w_{\beta, \xi}(x)$ holds on $\mathfrak{X}$ when $\sigma(x):=x$ and $\tau(x):=-(1-\xi)x+\beta\xi$. We define
    \[m_n:=\tau'n+\sigma''n(n-1)/2=-(1-\xi)n,\]
    which satisfies that $m_0=0>m_1>\cdots$. Moreover, $x^n\sigma(x)w_{\beta, \xi}(x)$ is equal to 0 at $x=0$ and tends to 0 as $x\to \infty$ for every $n\geq0$. Thus, for every $n\geq0$ there exists a polynomial solution of Equation \eqref{eq:hypergeometric} with $\lambda=-m_n$ and they are orthogonal with respect to $w_{\beta, \xi}$. We call them the \emph{Meixner polynomials}.
\end{example}

\begin{example}[Charlier polynomiasl]\label{ex:Charlier}
    Let $\mathfrak{X}=\mathbb{Z}_{\geq0}$ and $w_\mu\colon \mathfrak{X}\to \mathbb{R}_>0$ be defined by
    \[w_\mu(x):=\frac{e^{-\mu}\mu^x}{x!},\]
    where $\mu>0$. Namely, $w_\mu$ is the Poisson distribution with parameter $\lambda$. If we define $\sigma(x):=x$ and $\tau(x):=\mu-x$, then $\Delta[\sigma(x)w_\mu(x)]=\tau(x)w_\mu(x)$ holds on $\mathfrak{X}$. We also define
    \[m_n:=\tau'n+\sigma''n(n-1)/2=-n,\]
    which satisfies that $m_0=0>m_1>\cdots$. Moreover, $x^n\sigma(x)w_{\beta, \xi}(x)$ is equal to 0 at $x=0$ and tends to 0 as $x\to \infty$ for every $n\geq0$. Thus, for every $n\geq0$ there exists a polynomial solution of Equation \eqref{eq:hypergeometric} with $\lambda=-m_n$ and they are orthogonal with respect to $w_\mu$. We call them the \emph{Charlier polynomials}.
\end{example}

Following \cite{Olshanski03, BO05}, we introduce the following terminology: a pair of $z, z'\in\mathbb{C}$ is said to be
\begin{itemize}
    \item \emph{principal} if $z, z'\in\mathbb{C}\backslash\mathbb{R}$ and $z'=\overline{z}$,
    \item \emph{complementary} if $z, z'\in\mathbb{R}$ and $k<z, z'<k+1$ for some integer $k\in\mathbb{Z}$.
\end{itemize}
It is easy to check that a pair of $z, z'\in\mathbb{C}$ is either principal or complementary if and only if $(z+k)(z'+k)>0$ for any $k\in\mathbb{Z}$.

\begin{example}[Askey--Lesky polynomials]\label{ex:AL}
    Let $\mathfrak{X}=\mathbb{Z}$ and $u, u', w, w'\in\mathbb{C}$ such that both pairs of $u, u'$ and $w, w'$ are principal or complementary. We suppose that $u+u'+w+w'>2N+1$. A weight function $w:=w_{u, u', w, w'}\colon \mathfrak{X}\to \mathbb{R}_{>0}$ is defined by
    \[w(x):=\frac{1}{\Gamma(u-x+1)\Gamma(u'-x+1)\Gamma(w+x+1)\Gamma(w'+x+1)}.\]
    We remark that $w(x)$ is strictly positive on $\mathfrak{X}$ (see \cite[Lemma 7.9]{Olshanski03}). If we define
    \[\sigma(x):=(x+w)(x+w'),\quad \tau(x):=-(u+u'+w+w')x+uu'-ww',\]
    then $\Delta[\sigma(x)w(x)]=\tau(x)w(x)$ holds on $\mathfrak{X}$. We also define
    \[m_n:=\tau'n+\sigma''n(n-1)/2=-(u+u'+w+w')n + n(n-1),\]
    which satisfies that $m_0=0>m_1>\cdots>m_{N+1}$. Moreover, $x^n\sigma(x)w(x)\to 0$ as $x\to \pm\infty$ for $n=0, 1, \dots, 2N-2$. Thus, for $n=0, \dots, N-1$ there exists a polynomial solution of Equation \eqref{eq:hypergeometric} with $\lambda=-m_n$ and they are orthogonal with respect to $w$. Following \cite{Olshanski03, BO05}, we call them the \emph{Askey--Lesky polynomials}.
\end{example}

Let us explain a representation-theoretic origin of this example. See \cite{Olshanski03, BO05} for more details. It is well known that the irreducible representations of the unitary group $U(N)$ of rank $N$ are parametrized by these highest weights. More explicitly, they are parametrized by
\[\mathbb{S}_N:=\{\lambda=(\lambda_1\geq \cdots \lambda_N)\in\mathbb{Z}^N\}.\]
We recall that a \emph{character} of $U(N)$ is a continuous function on $U(N)$ that is positive-definite, central, and normalized (i.e., the value at the identity is equal to one). Since $U(N)$ is compact, every character can be decomposed into a convex combination of \emph{irreducible} characters, which are given by irreducible representations of $U(N)$. Namely, they are parametrized by $\mathbb{S}_N$. Hence, by taking coefficients of decompositions, characters of $U(N)$ associate with probability measures on $\mathbb{S}_N$. Moreover, there exists a bijection $\lambda\in\mathbb{S}_N\mapsto \omega^{(\lambda)}\in\mathcal{C}_N(\mathbb{Z})$ given by
\[\omega^{(\lambda)}_x:=\begin{cases}1 & x=\lambda_i+N-i \text{ for }i=1, \dots, N,\\ 0 & \text{otherwise}.\end{cases}\]
Therefore, providing $N$-point processes on $\mathbb{Z}$, including Example \ref{ex:AL}, is essentially equivalent to providing characters of $U(N)$.

The following is another type of difference operator discussed in \cite{Olshanski08}. It might confuse, but this is referred to as the \emph{hypergeometric} difference operator because its eigenfunctions are expressed by the Gauss hypergeometric function.
\begin{example}\label{ex:z-measure}
    Let $z, z'\in\mathbb{C}$ be a principal or complementary pair and $\xi\in(0, 1)$. We define a operator $D_{z, z', \xi}$ on $\mathfrak{X}:=\mathbb{Z}+1/2$ by
    \begin{align*}
        [D_{z, z', \xi}f](x) & =\sqrt{\xi(z+x+1/2)(z'+x+1/2)}f(x+1)-(x+\xi(z+z'+x))f(x) \\&\quad+\sqrt{\xi(z+x-1/2)(z'+x-1/2)}f(x-1)
    \end{align*}
    with domain $\ell^2_0(\mathfrak{X})$. By \cite[Proposition 2.2]{Olshanski08}, $D_{z, z', \xi}$ is essentially self-adjoint, and we use the same symbol $D_{z, z', \xi}$ to denote its closure. Moreover, the spectrum of $D_{z, z', \xi}$ consists of $\{(1-\xi)a \mid a\in \mathfrak{X}\}$, and there exists an orthonormal basis $\{\psi_{a; z, z', \xi}\}_{a\in\mathfrak{X}}$ such that
    \[D_{z, z', \xi}\psi_{a; z, z', \xi}=(1-\xi)a\psi_{a;z, z', \xi}.\]
    See \cite[Equation (2.1)]{BO06} for an explicit expression of $\psi_{a; z, z', \xi}$.
\end{example}
This example also has a representation-theoretic origin and relates to a determinantal point process. We will explain them later (see Section \ref{sec:hgdo}). In the regime
\[z, z'\to\infty, \quad \xi\to0, \quad zz'\xi\to \theta>0,\]
the hypergeometric difference operator $D_{z, z', \xi}$ becomes the difference operator $D_\theta^\text{Bessel}$ on $\mathbb{Z}+1/2$ defined by
\[[D_\theta^\text{Bessel}f](x)=\sqrt{\theta}f(x+1)-xf(x)+\sqrt{\theta}f(x-1).\]
It is known that there exists an orthonormal basis $\{\psi_{a; \theta}\}_{a\in\mathbb{Z}+1/2}$ such that $D_\theta^\text{Bessel}\psi_{a; \theta}=a\psi_{a; \theta}$. Moreover, $\psi_{a; \theta}$ can be expressed by the Bessel function. See \cite{Olshanski08} for more details.

\section{Dynamics of unitary operators}\label{sec:unitary_dynamics}
\subsection{General formalism}
In this section, we study unitary dynamics on determinantal point processes. First, we start with a basic observation of operator algebras.

Let $\mathcal{H}$ be a separable complex Hilbert space and $\mathfrak{I}(\mathcal{H})$ the GICAR algebra over $\mathcal{H}$. We recall that for all unitary operator $U$ on $\mathcal{H}$ there exists a $*$-automorphism $\gamma_U\colon \mathfrak{I}(\mathcal{H})\to \mathfrak{I}(\mathcal{H})$ such that $\gamma_U(a^*(h)a(k))=a^*(Uh)a(Uk)$ for any $h, k\in\mathcal{H}$. Let us fix a positive contraction operator $K$ on $\mathcal{H}$, and $\varphi_K$ denotes the corresponding quasi-free state on $\mathfrak{I}(\mathcal{H})$. We denote by $(\rho_K, \mathcal{H}_K, \Omega_K)$ the GNS-triple associated with $\varphi_K$. See Section \ref{subsec:GICAR} for its explicit realization. Since the GNS representation is unique up to unitary equivalence, we have the following statement:

\begin{proposition}\label{prop:second_quantization}
    Let $A$ be a self-adjoint operator on $\mathcal{H}$ and $(U_t)_{t\in\mathbb{R}}$ the strongly continuous one-parameter group of unitary operators generated by $A$, that is, $U_t=e^{\mathrm{i}tA}$ for any $t\in\mathbb{R}$. If $A$ commutes with $K$, then the following hold true:
    \begin{enumerate}
        \item There exists a unique strongly continuous one-parameter group $(\Gamma_K(U_t))_{t\in\mathbb{R}}$ of unitary operators on $\mathcal{H}_K$ such that $\Gamma_K(U_t)\rho_K(x)\Omega_K=\rho_K(\gamma_{U_t}(x))\Omega_K$ for any $x\in\mathfrak{I}(\mathcal{H})$.
        \item Let $d\Gamma_K(A)$ denote the infinitesimal generators of $(\Gamma_K(U_t))_{t\in\mathbb{R}}$. If $\mathcal{V}$ is a dense subspace in $\mathcal{H}$ such that $\mathcal{V}\subset \mathrm{dom}(A)$ and $U_t \mathcal{V}\subset \mathcal{V}$ for any $t\in\mathbb{R}$, then
              \[\widetilde{\mathcal{V}}:=\mathrm{span}\{\rho_K(a^*(h_1)\cdots a^*(h_n)a(k_1)\cdots a(k_n))\Omega_K\mid h_1, \dots, h_n, k_1, \dots ,k_n\in \mathcal{V}\}\]
              is a core of $d\Gamma_K(A)$.
    \end{enumerate}
\end{proposition}
\begin{proof}
    By the universality of CAR algebras, there exists an one-parameter group $(\gamma_{U_t})_{t\in\mathbb{R}}$ of $*$-automorphisms on $\mathfrak{A}(\mathcal{H})$ such that $\gamma_{U_t}(a(h))=a(U_th)$ for any $h\in \mathcal{H}$. Moreover, $(\gamma_{U_t})_{t\in\mathbb{R}}$ is strongly continuous since $\|a(h)\|=\|h\|$ for any $h\in\mathcal{H}$. We remark that $\mathfrak{I}(\mathcal{H})$ is invariant under $(\gamma_{U_t})_{t\in\mathbb{R}}$ since $\gamma_{U_t}\circ \gamma_\lambda=\gamma_\lambda\circ \gamma_{U_t}$ for any $t\in\mathbb{R}$ and $\lambda\in\mathbb{T}$. Furthermore, $\varphi_K\circ \gamma_{U_t}=\varphi_K$ holds for every $t\in\mathbb{R}$ by the definition of quasi-free states and the assumption that $U_tK=KU_t$. Thus, we obtain a strongly continuous one-parameter group $(\Gamma_K(U_t))_{t\in\mathbb{R}}$ of unitary operators on $\mathcal{H}_K$ such that
    \[\Gamma_K(U_t)\Omega_K=\Omega_K, \quad \Gamma_K(U_t)\rho_K(x)=\rho_K(\gamma_{U_t}(x))\Gamma_K(U_t)\]
    for every $t\in\mathbb{R}$ and $x\in\mathfrak{I}(\mathcal{H})$.

    If $h_1, \dots, h_n, k_1, \dots, k_n\in \mathrm{dom}(A)$, then $\rho_K(a^*(h_1)\cdots a^*(h_n)a(k_1)\cdots a(k_n))\Omega_K$ belongs to $\mathrm{dom}(d\Gamma_K(A))$ and we have
    \begin{align*}
         & d\Gamma_K(A)\rho_K(a^*(h_1)\cdots a^*(h_n)a(k_1)\cdots a(k_n))\Omega_K                       \\
         & =\sum_{i=1}^n\rho_K(a^*(h_1)\cdots a^*(Ah_i)\cdots a^*(h_n)a(k_1)\cdots a(k_n))\Omega_K      \\
         & \quad+\sum_{i=1}^n\rho_K(a^*(h_1)\cdots  a^*(h_n)a(k_1)\cdots a(Ak_i)\cdots a(k_n))\Omega_K.
    \end{align*}
    Thus, we have $\widetilde{\mathcal{V}}\subset \mathrm{dom}(d\Gamma_K(A))$. Moreover, $\widetilde{\mathcal{V}}$ is dense in $\mathcal{H}_K$ and invariant under $(\Gamma_K(U_t))_{t\in\mathbb{R}}$. Therefore, $\widetilde{\mathcal{V}}$ is a core for $d\Gamma_K(A)$ by \cite[Theorem VIII.11]{RS:book1}.
\end{proof}

Now we suppose that there exists an orthonormal basis $(v_\alpha)_{\alpha\in \mathcal{I}}$ for $\mathcal{H}$ such that $Av_\alpha=m_\alpha v_\alpha$ for some $m_\alpha\in\mathbb{R}$, where $\mathcal{I}$ is a certain countable index set. Let $\mathcal{I}_0\subset \mathcal{I}$ and $K$ denote the orthogonal projection from $\mathcal{H}$ onto $\overline{\mathrm{span}}\{v_\alpha\}_{\alpha\in \mathcal{I}_0}$. Since $K$ commutes with $A$, by Proposition \ref{prop:second_quantization}, we obtain the dynamics $(\Gamma_K(U_t))_{t\in\mathbb{R}}$ of unitary operators on $\mathcal{H}_K$, where $U_t=e^{\mathrm{i}tA}$ for any $t\in\mathbb{R}$. Its generator $d\Gamma_K(A)$ has a core given as follows.
\begin{lemma}\label{lem:core}
    For any $\alpha_1, \dots,\alpha_n\in\mathcal{I}\backslash \mathcal{I}_0$ and $\beta_1, \dots, \beta_n\in\mathcal{I}_0$ we define
    \[v_{\alpha_1, \dots, \alpha_n;\beta_1, \dots, \beta_n}:=\rho_K(a^*(v_{\alpha_1})\cdots a^*(v_{\alpha_n})a(v_{\beta_1})\cdots a(v_{\beta_n}))\Omega_K.\]
    Then $v_{\alpha_1, \dots, \alpha_n;\beta_1, \dots, \beta_n}\in \mathrm{dom}(d\Gamma_K(A))$ and the linear span of them is a core for $d\Gamma_K(A)$.
\end{lemma}
\begin{proof}
    Since $(v_\alpha)_{\alpha\in\mathcal{I}}$ is an orthonormal basis for $\mathcal{H}$, the linear span of them, denoted by $\mathcal{V}$, is dense in $\mathcal{H}$. Moreover, we have $U_tv_\alpha=e^{\mathrm{i}t m_\alpha}v_\alpha$, and hence $\mathcal{V}\subset \mathrm{dom}(A)$ and invariant under $(U_t)_{t\in\mathbb{R}}$. Thus, by Proposition \ref{prop:second_quantization}, the statement holds true.
\end{proof}

Since $U_tv_\alpha=e^{\mathrm{i}tm_\alpha}v_\alpha$ for any $t\geq0$ and $\alpha\in\mathcal{I}$, we have
\begin{align}\label{eq:action_U_t}
     & \Gamma_K(U_t)v_{\alpha_1, \dots, \alpha_n;\beta_1, \dots, \beta_n}\nonumber                                                   \\
     & =\rho_K(a^*(U_tv_{\alpha_1})\cdots a^*(U_tv_{\alpha_n})a(U_tv_{\beta_1})\cdots a(U_tv_{\beta_n}))\Omega_K\nonumber            \\
     & =\exp\left(\mathrm{i}t\sum_{j=1}^n(\mu_{\alpha_j}-\mu_{\beta_j})\right)v_{\alpha_1, \dots, \alpha_n;\beta_1, \dots, \beta_n}.
\end{align}
Thus, we have
\begin{equation}\label{eq:eigenvector}
    d\Gamma_K(A)v_{\alpha_1, \dots, \alpha_n;\beta_1, \dots, \beta_n}=\left(\sum_{j=1}^n(\mu_{\alpha_j}-\mu_{\beta_i})\right)v_{\alpha_1, \dots, \alpha_n;\beta_1, \dots, \beta_n}.
\end{equation}

As consequence of the above results, we obtain an expression of $d\Gamma_K(A)$ in terms of creation and annihilation operators.
\begin{theorem}\label{thm:hamiltonian}
    Let
    \[\mathcal{A}:=\sum_{\alpha\in\mathcal{I}\backslash\mathcal{I}_0}m_\alpha\rho_K(a^*(v_\alpha)a(v_\alpha))-\sum_{\alpha\in\mathcal{I}_0}m_\alpha\rho_K(a(v_\alpha)a^*(v_\alpha))\]
    defined on the linear span of the $v_{\alpha_1, \dots, \alpha_n;\beta_1, \dots, \beta_n}$. Then its colure coincides with $d\Gamma_K(A)$.
\end{theorem}
\begin{proof}
    By Lemma \ref{lem:core}, it suffices to show that $\mathcal{A}$ coincides with $d\Gamma_K(A)$ on the linear span of the $v_{\alpha_1, \dots, \alpha_n;\beta_1, \dots, \beta_n}$. If $\alpha\in \mathcal{I}_0$, then we have $\varphi_K((a(v_\alpha)a^*(v_\alpha))^*(a(v_\alpha)a^*(v_\alpha)))=0$, and hence $\rho_K(a(v_\alpha)a(v_\alpha)^*)\Omega_K=0$. By the canonical anticommutation relation, for any $\alpha\in\mathcal{I}_0$ we have
    \begin{align*}
         & a(v_\alpha)a^*(v_\alpha)a^*(v_{\alpha_1})\cdots a^*(v_{\alpha_n})a(v_{\beta_1})\cdots a(v_{\beta_n})                                                     \\
         & =a^*(v_{\alpha_1})\cdots a^*(v_{\alpha_n})a(v_{\beta_1})\cdots a(v_{\beta_n})\left(a(v_\alpha)a^*(v_\alpha)+\sum_{i=1}^n\delta_{\alpha, \beta_i}\right),
    \end{align*}
    and hence
    \[\left(\sum_{\alpha\in\mathcal{I}_0}m_\alpha\rho_K(a(v_\alpha)a^*(v_\alpha))\right)v_{\alpha_1, \dots, \alpha_n;\beta_1, \dots, \beta_n}=\left(\sum_{j=1}^nm_{\beta_j}\right)v_{\alpha_1, \dots, \alpha_n;\beta_1, \dots, \beta_n}.\]
    Similarly, $\rho_K(a^*(v_\alpha)a(v_\alpha))\Omega_K=0$ holds for every $\alpha\in\mathcal{I}\backslash\mathcal{I}_0$. By the canonical anticommutation relation, for any $\alpha\in\mathcal{I}\backslash \mathcal{I}_0$ we have
    \begin{align*}
         & a(v_\alpha)a^*(v_\alpha)a^*(v_{\alpha_1})\cdots a^*(v_{\alpha_n})a(v_{\beta_1})\cdots a(v_{\beta_n})                                                      \\
         & =a^*(v_{\alpha_1})\cdots a^*(v_{\alpha_n})a(v_{\beta_1})\cdots a(v_{\beta_n})\left(a(v_\alpha)a^*(v_\alpha)+\sum_{i=1}^n\delta_{\alpha, \alpha_i}\right),
    \end{align*}
    and hence
    \[\left(\sum_{\alpha\in\mathcal{I}\backslash\mathcal{I}_0} m_\alpha\rho_K(a^*(v_\alpha)a(v_\alpha))\right)v_{\alpha_1, \dots, \alpha_n;\beta_1, \dots, \beta_n}=\left(\sum_{j=1}^n\mu_{\alpha_j}\right)v_{\alpha_1, \dots, \alpha_n;\beta_1, \dots, \beta_n}.\]
    Therefore, by Equation \eqref{eq:eigenvector}, we have $\mathcal{A}v_{\alpha_1, \dots, \alpha_n;\beta_1, \dots, \beta_n}=d\Gamma_K(A)v_{\alpha_1, \dots, \alpha_n;\beta_1, \dots, \beta_n}$, and hence $\overline{\mathcal{A}}=d\Gamma_K(A)$ holds.
\end{proof}

For any $n\geq1$ we define $[n]:=\{1, \dots, n\}$.
\begin{remark}
    Evans \cite{Evans79} introduced the Wick ordered product with respect to quasi-free states. For any positive contraction operator $K$, the Wick ordered product with respect to $\varphi_K$ is defined as follows:
    \begin{align}\label{eq:Wick_ordering}
         & :a^*(h_1)\cdots a^*(h_m)a(k_n)\cdots a(k_1):_K \nonumber                                                                                                                \\
         & :=\sum_{I, J}\epsilon(I, J)\varphi_K(a^*(h_{i_p})\cdots a^*(h_{i_1})a(k_{j_1})\cdots a(k_{j_p}))a^*(h_{i'_1})\cdots a^*(h_{i'_{m-p}})a(k_{j'_{n-p}})\cdots a(k_{j'_1}),
    \end{align}
    where the summation runs over all $I=\{i_1<\cdots < i_p\}\subseteq [m]$ and $J=\{j_1<\cdots <j_p\}\subseteq [n]$ such that $|I|=|J|=p$, and we set $[m]\backslash I=\{i'_1<\cdots <i'_{m-p}\}$ and $[n]\backslash J=\{j'_1<\cdots < j'_{n-p}\}$. Moreover, $\epsilon(I, J)$ is the product of $(-1)^{(m-p)(n-p)}$ and the signs of two permutations
    \[(1, \dots, m)\mapsto (i_1, \dots, i_p, i'_1, \dots, i'_{m-p}), \quad (1, \dots, n)\mapsto (j_1, \dots, j_p, j'_1, \dots, j'_{n-p}).\]
    In particular, if $\{v_\alpha\}_{\alpha\in \mathcal{I}}$ and $K$ are the same as in Theorem \ref{thm:hamiltonian}, then we have
    \[:a^*(v_\alpha)a(v_\alpha):_K=\begin{cases}a(v_\alpha)a^*(v_\alpha)& \alpha\in \mathcal{I}_0, \\ -a^*(v_\alpha)a(v_\alpha) & \alpha\in\mathcal{I}\backslash \mathcal{I}_0.\end{cases}\]
    Thus, we have $\mathcal{A}=-\sum_{\alpha\in\mathcal{I}}m_\alpha \rho_K(:a^*(v_\alpha)a(v_\alpha):_K)$.
\end{remark}

\subsection{Discrete orthogonal polynomial ensembles of hypergeometric type}
We return to the setting of Section \ref{sec:hyper_dop}. Namely, let $\mathfrak{X}=\mathbb{Z}+a$ or $\mathbb{Z}_{\geq0}+a$ for some $a\in \mathbb{R}$, and a weight function $w\colon \mathfrak{X}\to\mathbb{R}_{>0}$ has finite moments of all orders $n\geq 0$. Moreover, we assume that there exist two polynomials $\sigma(x)$ and $\tau(x)$ such that $\deg\sigma(x)\geq2$, $\deg\tau(x)\leq1$, and
\[\Delta[\sigma(x)w(x)]=\tau(x)w(x), \quad \sigma(x)>0\]
holds on $\mathfrak{X}$. Thus, we obtain monic orthogonal polynomials $(\tilde p_n(x))_{n\geq0}$ satisfying Equation \eqref{eq:hypergeometric} with $\lambda=-m_n$ for all $n\geq0$, where $m_n:=\tau'n+\sigma''n(n-1)/2$. We also obtain the orthonormal basis $(p_n)_{n\geq0}$ for $\ell^2(\mathfrak{X})$ by $p_n(x):=\tilde p_n(x)w(x)^{1/2}/\|\tilde p_n\|_{L^2(\mathfrak{X}, w)}$. Then, they are eigenfunctions of $D$ defined by Equation \eqref{eq:differene_operator}. More precisely, we have $Dp_n=m_np_n$ for any $n\geq0$. We also assume that eigenvalues satisfy $m_0=0>m_1>m_2>\cdots$.

Let us fix $N\geq0$ and $K$ denote the orthogonal projection from $\ell^2(\mathfrak{X})$ onto the linear span of $p_0, \dots, p_{N-1}$. As we saw, $K$ is given by the normalized Christoffel--Darboux kernel (see Equation \eqref{eq:CD}), and this is a correlation kernel of the discrete orthogonal polynomial ensemble $M_{w, N}$. Since $K$ commutes with $D$, we obtain the one-parameter group $(\Gamma_K(U_t))_{t\in\mathbb{R}}$ on $\mathcal{H}_K$ by Proposition \ref{prop:second_quantization}, where $U_t=e^{\mathrm{i}tD}$. In this section, we analyze $(\Gamma_K(U_t))_{t\in\mathbb{R}}$ in terms of $L^2(\mathcal{C}_N(\mathfrak{X}), M_{w, N})$. First we show two isomorphisms $\mathcal{H}_K\cong\bigwedge^N\ell^2(\mathfrak{X})$ and $\bigwedge^N\ell^2(\mathfrak{X})\cong L^2(\mathcal{C}_N(\mathfrak{X}), M_{w, N})$.

In order to show them, we will use the idea of partitions of integers and their Frobenius coordinates. In the paper, a \emph{partition} is a sequence $\lambda=(\lambda_1\geq\lambda_2\geq\cdots)$ of nonnegative integers such that $|\lambda|:=\lambda_1+\lambda_2+\cdots <\infty$. By definition, we have $\lambda_l=0$ for sufficiently large $l\geq0$. The \emph{length} of $\lambda$ is defined as the integer $l$ such that $\lambda_l\neq 0$ and $\lambda_{l+1}=0$. We often write $\lambda=(\lambda_1, \dots, \lambda_l)$ when the length of $\lambda$ is $l$. Usually, a partition $\lambda$ is identified with its \emph{Young diagram}, which consists of boxes arranged in left-justified rows. The number of boxes in each row is $\lambda_1, \lambda_2, \dots$ from top to bottom. For instance, the partition $(5, 4, 4, 1)$ is represented by the following diagram:
\begin{center}
    \begin{tikzpicture}[scale=0.7]
        \draw (0, 0)--(5, 0)--(5, -1)--(4, -1)--(4, -3)--(1, -3)--(1, -4)--(0, -4)--(0,0);
        \draw (0, -1)--(4, -1);
        \draw (0, -2)--(4, -2);
        \draw (0, -3)--(1, -3);
        \draw (1, 0)--(1, -3);
        \draw (2, 0)--(2, -3);
        \draw (3, 0)--(3, -3);
        \draw (4, 0)--(4, -1);
    \end{tikzpicture}
\end{center}
As a convention, we use the symbol $\emptyset$ to denote the Young diagram with no boxes, or equivalently, the corresponding partition $(0, 0, \dots)$. Let $d$ denote the number of diagonal boxes. In the above example, we have $d=3$. For each row, the number of boxes to the right of the diagonal box is denoted by $\alpha_1, \alpha_2,\dots, \alpha_d$. Similarly, for each column, the number of boxes to the bottom of the diagonal box is denoted by $\beta_1, \beta_2,\dots, \beta_d$. These numbers $(\alpha|\beta)=(\alpha_1,\dots, \alpha_d|\beta_1, \dots, \beta_d)$ also describes the partition $\lambda$, and this is called the \emph{Frobenius coordinate} of $\lambda$. In the above example, the Frobenius coordinate $(\alpha|\beta)$ is $(4, 2, 1|3, 1, 0)$. See \cite[Chapter 1]{Macdonald} for more details.

Using the idea of Frobenius coordinates, we have the following:
\begin{lemma}\label{lem:first_isomorphism}
    There exists a unitary map $v\colon \mathcal{H}_K\to \bigwedge^N\ell^2(\mathfrak{X})$ such that $vp_{(\alpha| \beta)}=p_\lambda$ for any $\alpha_1>\cdots >\alpha_d\geq0$ and $N-1\geq \beta_1>\cdots >\beta_d\geq0$, where
    \[p_{(\alpha| \beta)}:=\rho_K(a^*(p_{\alpha_1+N})\cdots a^*(p_{\alpha_d+N})a(p_{N-1-\beta_1})\cdots a(p_{N-1-\beta_d}))\Omega_K,\]
    \[p_\lambda:=p_{\lambda_1+N-1}\wedge p_{\lambda_N+N-2}\wedge \cdots \wedge p_{\lambda_N}\]
    and $\lambda=(\lambda_1\geq\cdots\geq\lambda_N)$ is a partition whose Frobenius coordinate is $(\alpha|\beta)$.
\end{lemma}

Next, we show that $\bigwedge^N\ell^2(\mathfrak{X})\cong L^2(\mathcal{C}_N(\mathfrak{X}), M_{w, N})$. We remark that $M_{w, N}$ is $N$-point process, that is, its support is contained in $\mathcal{C}_N(\mathfrak{X})$. We define
\[\mathfrak{X}^{(N)}:=\{(x_1, \dots, x_N)\in \mathfrak{X}^N\mid x_1<\cdots < x_N\},\]
and $\mathfrak{X}^{(N)}$ is identified with $\mathcal{C}_N(\mathfrak{X})$ by the correspondence $(x_1, \dots, x_N)\in\mathfrak{X}^{(N)}\leftrightarrow \omega\in\mathcal{C}_N(\mathfrak{X})$ given as $\omega_x=1$ only if $x=x_1, \dots, x_N$. Under the identification $\mathfrak{X}^{(N)}\cong \mathcal{C}_N(\mathfrak{X})$, we also identify $M_{w, N}$ with the corresponding probability measure on $\mathfrak{X}^{(N)}$. Namely, our goal here is to show $\bigwedge^N\ell^2(\mathfrak{X})\cong L^2(\mathfrak{X}^{(N)}, M_{w, N})$ (see Lemma \ref{lem:second_isomorphism}).

\begin{definition}\label{def:second_isomorphism}
    For any $h_1, \dots, h_N\in \ell^2(\mathfrak{X})$ we define a function $F_{h_1, \dots, h_N}$ on $\mathfrak{X}^{(N)}$ by
    \[F_{h_1, \dots, h_N}(x_1, \dots, x_N):=\frac{{Z_{w, N}^{1/2}}}{\prod_{j=1}^Nw(x_j)^{1/2}}\frac{\det[h_i(x_j)]_{i, j=1}^N}{V_N(x_1, \dots, x_N)}\]
    for any $(x_1, \dots, x_N)\in\mathfrak{X}^{(N)}$, where $V_N(x_1, \dots, x_N):=\prod_{1\leq i<j\leq N}(x_i-x_j)$. In particular, we define $F_\lambda:=F_{p_{\lambda_1+N-1}, p_{\lambda_2+N-2}, \dots, p_{\lambda_N}}$ for any partition $\lambda=(\lambda_1\geq \cdots \geq\lambda_N)$ with length $\leq N$.
\end{definition}
We remark that $\det[h_i(x_j)]_{i, j=1}^N$ and $V_N(x_1, \dots, x_N)$ in the right-hand side depend on the order of $x_1, \dots, x_N$, but $F_{h_1, \dots, h_N}$ itself is independent of the order of them. On the other hand, $F_{h_1, \dots, h_N}$ is anti-symmetric in $h_1, \dots, h_N$, and hence this is determined by $h_1\wedge \cdots \wedge h_N$. We give a few examples as follows.

\begin{example}\label{ex:constant_one}
    We recall that $Z_{w, N}$ is the normalization constant of $M_{w, N}$ (see Equation \eqref{eq:normalized_constant}), and it is equal to $\|\tilde p_0\|_{L^2(\mathfrak{X}, w)}\cdots \|\tilde p_{N-1}\|_{L^2(\mathfrak{X}, w)}$. Since $\tilde p_{i-1}(x)$ is a monic polynomial of degree $i-1$, we have $\det[\tilde p_{i-1}(x_j)]_{i, j=1}^N=V_N(x_1, \dots, x_N)$. Therefore, we have $F_\emptyset=F_{p_{N-1}, \dots, p_0}\equiv 1$.
\end{example}

\begin{example}
    Let $(e_x)_{x\in \mathfrak{X}}$ be the canonical basis for $\ell^2(\mathfrak{X})$. We have
    \[F_{e_{x_1}, \dots, e_{x_N}}(y_1, \dots, y_N)=\begin{cases}Z_{w, N}^{1/2}V_N(x_1, \dots, x_N)^{-1}\prod_{j=1}^Nw(x_j)^{-1/2} & (y_i\text{'s in } \{x_1, \dots, x_N\})\\0& (\text{otherwise}.)\end{cases}\]
\end{example}

\begin{lemma}\label{lem:second_isomorphism}
    For any $h_1, \dots, h_N\in \ell^2(\mathfrak{X})$ the $F_{h_1, \dots, h_N}$ lies in $L^2(\mathfrak{X}^{(N)}, M_w)$, and the linear mapping $F\colon h_1\wedge \cdots \wedge h_N\in\bigwedge \ell^2(\mathfrak{X})\mapsto F_{h_1, \dots, h_N}\in L^2(\mathfrak{X}^{(N)}, M_w)$ is unitary.
\end{lemma}
\begin{proof}
    By the above example, the range of $F$ is dense in $L^2(\mathfrak{X}^{(N)}, M_{w, N})$. Thus, it suffices to show that $F$ is isometry. For any $h_1, \dots, h_N, k_1, \dots, k_N\in \ell^2(\mathfrak{X})$ we have
    \begin{align*}
        \langle F_{h_1, \dots, h_N}, F_{k_1, \dots, k_N}\rangle_{L^2(\mathfrak{X}^{(N)}, M_{w, N})}
         & =\sum_{\underline{x}\in\mathfrak{X}^{(N)}}F_{h_1, \dots, h_N}(\underline{x})\overline{F_{k_1, \dots, k_N}(\underline{x})}M_w(\underline{x})                 \\
         & =\sum_{\underline{x}\in\mathfrak{X}^{(N)}}\det[h_i(x_j)]_{i, j=1}^N\det[\overline{k_i(x_j)}]_{i, j=1}^N                                                     \\
         & =\frac{1}{N!}\sum_{x_1, \dots, x_N\in\mathfrak{X}}\det[h_i(x_j)]_{i, j=1}^N\det[\overline{k_i(x_j)}]_{i, j=1}^N                                             \\
         & =\frac{1}{N!}\sum_{\sigma, \tau\in S(N)}\mathrm{sgn}(\sigma)\mathrm{sgn}(\tau)\prod_{i=1}^N\langle h_{\sigma(i)}, k_{\tau(i)}\rangle_{\ell^2(\mathfrak{X})} \\
         & =\frac{1}{N!}\sum_{\sigma\in S(N)}\mathrm{sgn}(\sigma)\det[\langle h_{\sigma(i)}, k_j\rangle_{\ell^2(\mathfrak{X})}]_{i, j=1}^N                             \\
         & = \det[\langle h_i, k_j\rangle_{\ell^2(\mathfrak{X})}]_{i, j=1}^N                                                                                           \\
         & = \langle h_1\wedge \cdots \wedge h_N, k_1\wedge \cdots \wedge k_N\rangle.
    \end{align*}
\end{proof}

Composing two unitary maps $v$, $F$ given in Lemmas \ref{lem:first_isomorphism}, \ref{lem:second_isomorphism}, we obtain the unitary map $V:=F\circ v\colon \mathcal{H}_K\to L^2(\mathfrak{X}^{(N)}, M_{w, N})$. Let us recall that $\mathcal{H}_K$ admits the $*$-representation $(\rho_K, \mathcal{H}_K)$ of $\mathfrak{I}(\mathfrak{X}):=\mathfrak{I}(\ell^2(\mathfrak{X}))$. Hence, $L^2(\mathfrak{X}^{(N)}, M_{w, N})$ also admits a $*$-representation of $\mathfrak{I}(\mathfrak{X})$ given by $\pi_K:=\mathrm{Ad}V\circ \rho_K$. Moreover, since $\mathbbm{1}=F_\emptyset=V\Omega_K$, for any $x\in\mathfrak{I}(\mathfrak{X})$ we have
\[\langle \pi_K(x)\mathbbm{1}, \mathbbm{1}\rangle=\langle \rho_K(x)\Omega_K, \Omega_K\rangle=\varphi_K(x),\]
where $\mathbbm{1}$ is the constant function equal to 1.

\begin{remark}
    The isomorphism $\mathcal{H}_K\cong L^2(\mathfrak{X}^{(N)}, M_{w, N})$ is implicitly suggested in \cite{Olshanski20}. In the paper \cite{Olshanski20}, Olshanski introduced the concept of \emph{perfectness} for determinantal point processes. Roughly speaking, a determinantal point process $M$ on $\mathfrak{X}$ is said to be perfect if $L^2(\mathcal{C}(\mathfrak{X}), M)$ admits the GNS representation of $\mathfrak{I}(\mathfrak{X})$ associated with the quasi-free state obtained from the correlation kernel of $M$. The class of perfect measures is strictly smaller than the class of determinantal point processes. However, Olshanski showed that discrete orthogonal polynomial ensembles (on $\mathfrak{X}\subset\mathbb{R}$) are perfect. On the other hand, the GNS representation is unique up to unitary equivalent. Hence, our GNS space $\mathcal{H}_K$ should be isomorphic to $L^2(\mathfrak{X}^{(N)}, M_{w, N})$.
\end{remark}

By Equation \eqref{eq:action_U_t}, \eqref{eq:eigenvector}, the vectors $p_{(\alpha|\beta)}$ (in Lemma \ref{lem:first_isomorphism}) are eigenvectors of $\Gamma_K(U_t)$ and $d\Gamma_K(D)$, where $U_t=e^{\mathrm{i}t D}$. In what follows, we use the same symbols $\Gamma_K(U_t)$ and $d\Gamma_K(D)$ to denote $V\Gamma_K(U_t)V^*$ and $Vd\Gamma_K(D)V^*$ acting on $L^2(\mathfrak{X}^{(N)}, M_{w, N})$, respectively. They have eigenfunctions $F_\lambda$ for all $\lambda=(\lambda_1\geq\cdots \geq \lambda_N)$. The following is useful to describe those eigenvalues in terms of partitions $\lambda$.

\begin{lemma}\label{lem:eigenvalue}
    Let $\lambda=(\lambda_1\geq\cdots \geq \lambda_N)$ be a partition and $m_\lambda:=\sum_{i=1}^Nm_{\lambda_i+N-i}$. If the Frobenius coordinate of $\lambda$ is $(\alpha|\beta)=(\alpha_1, \dots, \alpha_n| \beta_1, \dots, \beta_n)$, then the following holds true:
    \[m_\lambda-m_\emptyset=\sum_{i=1}^n(m_{\alpha_i+N}-m_{N-1-\beta_i}).\]
\end{lemma}
\begin{proof}
    Let $\alpha_i':=\alpha_i+1/2$ and $\beta_i':=\beta_i+1/2$, where $(\alpha'|\beta')=(\alpha_1', \dots, \alpha_n'|\beta_1', \dots, \beta_n')$ is called the \emph{modified} Frobenius coordinate. By definition, we have $\sum_{i=1}^n(\alpha_i'+\beta_i')=|\lambda|:=\sum_{i=1}^N\lambda_i$. Moreover, by \cite[Proposition 6.7]{BO:book}, $\sum_{i=1}^n(\alpha_i'^2-\beta_i'^2)=\sum_{i=1}^N\left\{\left(\lambda_i-i+\frac{1}{2}\right)^2-\left(-i+\frac{1}{2}\right)^2\right\}$ holds. Thus, we have
    \begin{align*}
         & \sum_{i=1}^n(m_{\alpha_i+N}-m_{N-1-\beta_i})                                                                                                                                                                 \\
         & =\frac{\alpha''}{2}\sum_{i=1}^n\left\{\left(\alpha_i'+N-\frac{1}{2}\right)\left(\alpha_i'+N-\frac{3}{2}\right)-\left(\beta_i'-N+\frac{1}{2}\right)\left(\beta_i'-N+\frac{3}{2}\right)\right\}+\tau'|\lambda| \\
         & = \frac{\alpha''}{2}\sum_{i=1}^n(\alpha_i'^2-\beta_i'^2)+(\alpha''(N-1)+\tau')|\lambda|                                                                                                                      \\
         & =\frac{\alpha''}{2}\sum_{i=1}^N\left\{\left(\lambda_i-i+\frac{1}{2}\right)^2-\left(-i+\frac{1}{2}\right)^2\right\}+(\alpha''(N-1)+\tau')\sum_{i=1}^N\lambda_i                                                \\
         & =\sum_{i=1}^N(m_{\lambda_i+N-i}-m_{N-i}).
    \end{align*}
\end{proof}

By Lemma \ref{lem:eigenvalue} and Equations \eqref{eq:action_U_t}, \eqref{eq:eigenvector}, we obtain the following:

\begin{proposition}\label{prop:capsule}
    For any partition $\lambda$ and $t\in\mathbb{R}$ we have
    \[\Gamma_K(U_t)F_\lambda=e^{\mathrm{i}t(m_\lambda-m_\emptyset)}F_\lambda,\quad d\Gamma_K(D)F_\lambda=(m_\lambda-m_\emptyset) F_\lambda.\]
\end{proposition}

We remark that, by Lemma \ref{lem:core}, the linear span of the functions $F_\lambda$ is a core for $d\Gamma_K(D)$. Moreover, as a consequence of Theorem \ref{thm:hamiltonian}, we can describe $d\Gamma_K(D)$ in terms of creation and annihilation operators as follows:
\begin{corollary}\label{cor:orth_poly_cre_ann}
    Let
    \[\mathcal{A}:=\sum_{n=N}^\infty m_n\pi_K(a^*(p_n)a(p_n))-\sum_{n=0}^{N-1}m_n\pi_K(a(p_n)a^*(p_n))\]
    defined on the linear span of the $p_{(\alpha|\beta)}$ in Lemma \ref{lem:first_isomorphism}. Then its closure coincides with $d\Gamma_K(D)$.
\end{corollary}

\subsection{The hypergeometric difference operator}\label{sec:hgdo}
Here we investigate a dynamics generated by the hyepergeometic difference operator $D_{z, z', \xi}$ in Example \ref{ex:z-measure}. Before that, we discuss its relationship to representation theory and determinantal point processes. Let $\mathbb{Y}$ be the set of all partitions. As explained previously, partitions are identified with Young diagrams. We take parameters $z, z'$, and $\xi$ as the same as in Example \ref{ex:z-measure}. We define a probability measure $M_{z, z', \xi}$, called \emph{$z$-measure}, on $\mathbb{Y}$ by
\[M_{z, z', \xi}(\lambda):=(1-\xi)^{zz'}\xi^{|\lambda|}(z)_\lambda (z')_{\lambda}\left(\frac{\dim \lambda}{|\lambda|!}\right)^2\quad (\lambda\in\mathbb{Y}),\]
where $(z)_\lambda:=\prod_{i=1}^l\prod_{j=1}^{\lambda_i}(z+j-i)$ when $l$ is the length of $\lambda$. Moreover, $\dim \lambda$ is the number of the standard Young tableaux of the shape $\lambda$, and it coincides with the dimension of the irreducible representation of the symmetric group $S(|\lambda|)$ labeled by $\lambda$. Since a pair of $z, z'$ is principal or complementary, we have $zz'>0$ and $(z)_\lambda (z')_\lambda>0$. Thus, $M_{z, z', \xi}$ is strictly positive, that is, $M_{z, z', \xi}(\lambda)>0$ holds for every $\lambda\in \mathbb{Y}$.

Let $\mathbb{Z}':=\mathbb{Z}+1/2$ following a convention, and $\lambda\in\mathbb{Y}\mapsto \omega^\lambda\in\mathcal{C}(\mathbb{Z}')$ is defined by
\[\omega^\lambda_x=\begin{cases}1 & \text{if }x=\lambda_i-i+1/2 \text{ for }i=1, 2, \dots, \\0&\text{otherwise}.\end{cases}\]
The pushforward measure of $M_{z, z', \xi}$ under this map becomes a point process on $\mathbb{Z}'$, and this is known to be determinantal. Moreover, its correlation kernel $K_{z, z', \xi}$ is given by
\[K_{z, z', \xi}(x, y):=\sum_{a\in \mathbb{Z'}_+}\psi_a(x)\psi_a(y) \quad (x, y\in\mathbb{Z}'),\]
where $\mathbb{Z}'_+:=\mathbb{Z}_{\geq 0}+1/2$ and $\psi_a:=\psi_{a; z, z', \xi}$ ($a\in\mathbb{Z}'$). They form orthonormal basis for $\ell^2(\mathbb{Z}')$ and satisfy $D_{z, z', \xi}\psi_a=(1-\xi)a\psi_a$. Thus, the integral operator of $K_{z, z', \xi}$, denoted by $K$, is the spectral projection of $D_{z, z', \xi}$ corresponding to $\{(1-\xi)a\mid a\in\mathbb{Z}'_+\}$.

Since $K$ and $D_{z, z', \xi}$ commute, by Proposition \ref{prop:second_quantization}, we obtain a dynamics $(\Gamma_K(U_t))_{t\in\mathbb{R}}$ of unitary operators on $\mathcal{H}_K$, where $U_t:=e^{\mathrm{i}t D_{z, z', \xi}}$. Let $d\Gamma_K(D_{z, z', \xi})$ denote its generator. We denote by $(\rho_K, \mathcal{H}_K, \Omega_K)$ the GNS-triple associated with $\varphi_K$ on $\mathfrak{I}(\mathbb{Z}')$

\begin{proposition}\label{prop:ocha}
    Let $\psi_{\alpha_1, \dots, \alpha_n; \beta_1, \dots, \beta_n}:=\rho_K(a^*(\psi_{-\alpha_1})\cdots a^*(\psi_{-\alpha_n})a(\psi_{\beta_1})\cdots a(\psi_{\beta_n}))\Omega_K$ for any $\alpha_1>\cdots>\alpha_n\geq 1/2$ and $\beta_1>\dots >\beta_n\geq 1/2$. Then the linear span of them is a core of $d\Gamma(D_{z, z', \xi})$. If $\lambda$ is a partition with \emph{modified} Frobenius coordinate is $(\alpha_1, \dots, \alpha_n|\beta_1, \dots, \beta_n)$, then we have
    \[\Gamma_K(U_t)\psi_{\alpha_1, \dots, \alpha_n; \beta_1, \dots, \beta_n}=e^{-\mathrm{i}t(1-\xi)|\lambda|}\psi_{\alpha_1, \dots, \alpha_n; \beta_1, \dots, \beta_n},\]
    \[d\Gamma_K(D_{z, z', \xi})\psi_{\alpha_1, \dots, \alpha_n; \beta_1, \dots, \beta_n}=-(1-\xi)|\lambda|\psi_{\alpha_1, \dots, \alpha_n; \beta_1, \dots, \beta_n}.\]
\end{proposition}
\begin{proof}
    By Lemma \ref{lem:core}, the linear span of the $\psi_{\alpha_1, \dots, \alpha_n; \beta_1, \dots, \beta_n}$ is a core of $d\Gamma_K(D_{z, z', \xi})$. If $\lambda$ is a Young diagram with modified Frobenius coordinate is $(\alpha_1, \dots, \alpha_n|\beta_1, \dots, \beta_n)$, then we have $|\lambda|=\sum_{j=1}^n(\alpha_j+\beta_j)$ and hence, by Equation \eqref{eq:action_U_t}, \eqref{eq:eigenvector}, we obtain two formulas in the statement.
\end{proof}

Let $(f_\lambda)_{\lambda\in\mathbb{Y}}$ be an arbitrary orthonormal basis for $\ell^2(\mathbb{Y})$, labeled by partitions. We define an isomorphism $\mathcal{H}_K\cong \ell^2(\mathbb{Y})$ assigning $\psi_{\alpha_1, \dots, \alpha_n; \beta_1, \dots, \beta_n}$ to $f_\lambda$, where $(\alpha_1, \dots, \alpha_n | \beta_1, \dots, \beta_n)$ is a modified Frobenius coordinate of $\lambda$. We also consider the isomorphism $\ell^2(\mathbb{Y})\cong L^2(\mathbb{Y}, M_{z, z', \xi})$ defined as multiplication by $M_{z, z', \xi}^{-1/2}$. Combining them, we have $\mathcal{H}_K\cong L^2(\mathbb{Y}, M_{z, z', \xi})$ and operators $\Gamma_K(U_t)$ and $d\Gamma_K(D_{z, z', \xi})$ give operators on $L^2(\mathbb{Y}, M_{z, z', \xi})$, denoted by the same symbols, respectively. For any $\lambda\in \mathbb{Y}$ let $F_\lambda:=f_\lambda M_{z, z', \xi}^{-1}$ and we have
\[\Gamma_K(U_t)F_\lambda=e^{-\mathrm{i}t(1-\xi)|\lambda|}F_\lambda, \quad d\Gamma_K(D_{z, z', \xi})F_\lambda=-(1-\xi)|\lambda|F_\lambda.\]
We remark that, by Lemma \ref{lem:core}, the linear span of the functions $F_\lambda$ is a core for $d\Gamma_K(D_{z, z', \xi})$. Moreover, as a consequence of Theorem \ref{thm:hamiltonian}, we can describe $d\Gamma_K(D_{z, z', \xi})$ in terms of creation and annihilation operators as follows:
\begin{corollary}\label{cor:hyp_cre_ann}
    Let
    \[\mathcal{A}_{z, z', \xi}:=\sum_{a=-1/2}^{-\infty} (1-\xi)a\pi_K(a^*(\psi_a)a(\psi_a))-\sum_{n=1/2}^\infty (1-\xi)a\pi_K(a(\psi_a)a^*(\psi_a))\]
    defined on the linear span of the $\psi_{\alpha_1,\dots, \alpha_n;\beta_1, \dots, \beta_n}$ in Proposition \ref{prop:ocha}. Then its closure coincides with $d\Gamma_K(D_{z, z', \xi})$.
\end{corollary}

\section{Stochastic dynamics}\label{sec:stochastic_dynamics}
\subsection{General formalism}
In this section, we study stochastic dynamics on determinantal point processes. Similar to the previous section, we start with a basic observation of operator algebras.

Let $T\colon \mathcal{H}\to \mathcal{K}$ be an isometry of two Hilbert spaces $\mathcal{H}$ and $\mathcal{K}$. By the universality of CAR algebras, we obtain a $*$-homomorphism $a(T)\colon \mathfrak{A}(\mathcal{H})\to \mathfrak{A}(\mathcal{K})$ such that $a(T)(a(h))=a(Th)$ for any $h\in \mathcal{H}$. In particular, a scalar unitary operators $\lambda1$ ($\lambda\in\mathbb{T}$) induce the gauge action $\gamma\colon\mathbb{T}\curvearrowright \mathfrak{A}(\mathcal{H})$. Evans \cite{Evans79} had extended the construction of $a(T)$ to a construction of completely positive maps from contraction operators. We summarize this here. See \cite{Evans79,EK:book} for more details.

Let us fix a positive contraction operator $K$ on $\mathcal{H}$, and $\varphi_K$ denotes the corresponding quasi-free state on $\mathfrak{A}(\mathcal{H})$. We denote by $(\rho_K, \mathcal{F}_K(\mathcal{H}), \Omega_K)$ the GNS-triple associated with $\varphi_K$. Moreover, $\mathcal{H}_K\subset \mathcal{F}_K(\mathcal{H})$ denotes the closed linear span of $\rho_K(\mathfrak{I}(\mathcal{H}))\Omega_K$, where $\mathfrak{I}(\mathcal{H})$ is the GICAR algebra, i.e., $(\rho_K, \mathcal{H}_K, \Omega_K)$ is the GNS-triple associated with $\varphi_K$ on $\mathfrak{I}(\mathcal{H})$.

Let $T$ be a contraction operator on $\mathcal{H}$. We consider the isometry $W\colon \mathcal{H}\to \mathcal{H}\oplus\mathcal{H}$ defined by $Wh:=Th\oplus Sh$, where $S:=(1-T^*T)^{1/2}$. As mentioned above, we obtain the $*$-homomorphism $a(W)\colon \mathfrak{A}(\mathcal{H})\to \mathfrak{A}(\mathcal{H}\oplus \mathcal{H})$ such that $a(W)a(h)=a(Wh)$ for all $h\in\mathcal{H}$. Let us recall that the gauge action $\gamma\colon\mathbb{T}\curvearrowright \mathfrak{A}(\mathcal{H})$ preserves $\varphi_K$, that is, $\varphi_K\circ \gamma_\lambda=\varphi_K$ for any $\lambda\in\mathbb{T}$. Thus, there exists a unique unitary representation $(\Gamma, \mathcal{F}_K(\mathcal{H}))$ of $\mathbb{T}$ such that
\[\Gamma_\lambda\rho_K(x)=\rho_K(\gamma_\lambda(x))\Gamma_\lambda, \quad \Gamma_\lambda\Omega_K=\Omega_K \quad (\lambda\in\mathbb{T}, x\in\mathfrak{A}(\mathcal{H})).\]
We consider the $*$-homomorphism $\Phi_K\colon \mathfrak{A}(\mathcal{H}\oplus\mathcal{H})\to \mathfrak{A}(\mathcal{H})\otimes B(\mathcal{F}_K(\mathcal{H}))$ defined by
\[\Phi_K(a(h\oplus k)):=a(h)\otimes \Gamma_{-1}+1\otimes \rho_K(a(k)).\]
Then, we define a linear map $a_K(T)\colon \mathfrak{A}(\mathcal{H})\to \mathfrak{A}(\mathcal{H})$ by $a_K(T):=(\mathrm{id}\otimes \widetilde\varphi_K)\circ \Phi_K\circ a(W)$, where $\widetilde\varphi_K$ denotes the vector state on $B(\mathcal{F}_K(\mathcal{H}))$ corresponding to $\Omega_K$. By definition, $a_K(T)$ is contractive and completely positive. Moreover, $\gamma_\lambda\circ a_K(T)=a_K(T)\circ \gamma_\lambda$ holds for every $\lambda\in\mathbb{T}$, and hence $a_K(T)$ preserves the GICAR algebra $\mathfrak{I}(\mathcal{H})$. Namely, we have $a_K(T)(\mathfrak{I}(\mathcal{H}))\subseteq \mathfrak{I}(\mathcal{H})$.

By the definition of $a_K(T)$, for any $h_1, \dots, h_n\in\mathcal{H}$ we have
\begin{align*}
     & a_K(T)(a^{\#_1}(h_1)\cdots a^{\#_n}(h_n))                                                                                                                       \\
     & =\sum_{I\subseteq [n]}\epsilon_I \varphi_K(a^{\#_{i_1}}(Sh_{i_1})\cdots a^{\#_{i_p}}(Sh_{i_p}))a^{\#_{i'_1}}(Th_{i'_1})\cdots a^{\#_{i'_{n-p}}}(Th_{i'_{n-p}}),
\end{align*}
where $I=\{i_1<\cdots <i_p\}$, $[n]\backslash I=\{i'_1<\dots<i'_{n-p}\}$, and $\epsilon_I$ is the sign of the permutation $(1, \dots, n)\to (i_1, \dots, i_p, i'_1, \dots, i'_{n-p})$. Here $a^{\#_i}$ denotes either $a$ or $a^*$. Moreover, if $TK=KT$, then, by \cite[Theorem 3.1, Proposition A.1]{Evans79}, we have
\begin{equation}\label{eq:qf_map}
    a_K(T)(:a^{\#_1}(h_1)\cdots a^{\#_n}(h_n):_K)=:a^{\#_1}(Th_1)\cdots a^{\#_n}(Th_n):_K.
\end{equation}

The following result might be known, but we show this for the reader's convenience.
\begin{lemma}
    For any $h_1, \dots, h_m, k_1, \dots, k_n\in\mathcal{H}$ we have
    \begin{align}\label{eq:qf_map_calc}
         & a_K(T)(a^*(h_m)\cdots a^*(h_1)a(k_1)\cdots a(k_n))                                                                                                        \nonumber \\
         & =\sum_{I, J}\epsilon_I \epsilon_J\det[\langle KSh_{i_a}, Sk_{j_b}\rangle]_{a, b=1}^p a^*(Th_{i'_{m-p}})\cdots a^*(Th_{i'_1})a(Tk_{j'_1})\cdots  a(Tk_{j'_{n-p}}),
    \end{align}
    where the summation runs over all $I=\{i_1<\cdots <i_p\}\subseteq [m]$ and $J=\{j_1<\cdots <j_p\}\subseteq [n]$ such that $|I|=|J|(=p)$, and we set $[m]\backslash I=\{i'_1<\cdots <i'_{m-p}\}$ and $[n]\backslash J=\{j'_1<\cdots <j'_{n-p}\}$. Furthermore, $\varphi_K\circ a_K(T)=\varphi_K$ holds if $TK=KT$.
\end{lemma}
\begin{proof}
    We first show Equation \eqref{eq:qf_map_calc}. By the definition of $\Phi_K$, we have
    \[\Phi_K(a(Wk_1)\cdots a(Wk_n))=\sum_{J\subseteq [n]}\epsilon_Ja(Tk_{j'_1})\cdots a(Tk_{j'_{n-|J|}})\otimes \pi_K(a(Sk_{j_1})\cdots a(Sk_{j_{|J|}}))\Gamma_{-1}^{n-|J|},\]
    and hence
    \begin{align*}
         & \Phi_K(a^*(Wh_m)\cdots a^*(Wh_1)a(Wk_1)\cdots a(Wk_n))                                                                                            \\
         & =\Phi_K(a(Wh_1)\cdots a(Wh_m))^*\Phi_K(a(Wk_1)\cdots a(Wk_n))                                                                                     \\
         & =\sum_{I\subseteq [m], J\subseteq[n]}\epsilon_I\epsilon_J [a(Th_{i'_1})\cdots a(Th_{i'_{m-|I|}})]^*[a(Tk_{j'_1})\cdots a(Tk_{j'_{n-|J|}})]\otimes \\
         & \quad \quad \quad  [\pi_K(a(Sh_{i_1})\cdots a(Sh_{i_{|I|}}))\Gamma_{-1}^{m-|I|}]^*[\pi_K(a(Sk_{j_1})\cdots a(Sk_{j_{|J|}}))\Gamma_{-1}^{n-|J|}].
    \end{align*}
    Thus, Equation \eqref{eq:qf_map_calc} holds true since
    \begin{align*}
         & \widetilde \varphi_K([\pi_K(a(Sh_{i_1})\cdots a(Sh_{i_{|I|}}))\Gamma_{-1}^{m-|I|}]^*[\pi_K(a(Sk_{j_1})\cdots a(Sk_{j_{|J|}}))\Gamma_{-1}^{n-|J|}]) \\
         & =\delta_{|I|, |J|}\det[\langle KSh_{i_a}, Sk_{j_b}]_{a, b=1}^{|I|}.
    \end{align*}
    Next, we show the second assertion. By Equation \eqref{eq:qf_map_calc}, we have
    \begin{align*}
         & \varphi_K(a_K(T)[a^*(h_m)\cdots a^*(h_1)a(k_1)\cdots a(k_n)])                                                                                                                     \\
         & =\delta_{m, n}\sum_{I, J\subseteq [n]; |I|=|J|(=p)}\epsilon_I\epsilon_J\det[\langle KSh_{i_a}, Sk_{j_b}\rangle]_{a, b=1}^p\det[\langle KTh_{n-i'_a}, Tj_{n-j'_b}]_{a, b=1}^{n-p}.
    \end{align*}
    Since the sign of the permutation $(1, \dots, m)\to (i_1, \dots, i_p, i'_1, \dots, i'_{m-p})$ (i.e., $\epsilon_I$) is equal to
    \[(-1)^{|\{(j, j')\in J\times J'\mid j'<j\}|}=(-1)^{\sum_{a=1}^p(n-j_a+p-a)}.\]
    Therefore, by \cite[Equation (1)]{Marcus}, we have
    \begin{align*}
         & \varphi_K(a_K(T)[a^*(h_m)\cdots a^*(h_1)a(k_1)\cdots a(k_n)])                                                                                                                                    \\
         & =\delta_{m, n}\sum_{I, J\subseteq [n]; |I|=|J|(=p)}(-1)^{\sum_{i\in I}i+\sum_{j\in J}j}\det[\langle KSh_{i_a}, Sk_{j_b}\rangle]_{a, b=1}^p\det[\langle KTh_{n-i'_a}, Tj_{n-j'_b}]_{a, b=1}^{n-p} \\
         & =\delta_{m, n}\det[\langle KTh_i, Th_j \rangle+\langle KSh_i, Sk_j \rangle]_{i, j=1}^n.
    \end{align*}
    Thus, if $TK=KT$, then $T^*KT+S^*KS=K$, and hence we have
    \[\varphi_K(a_K(T)[a^*(h_m)\cdots a^*(h_1)a(k_1)\cdots a(k_n)])=\varphi_K(a^*(h_m)\cdots a^*(h_1)a(k_1)\cdots a(k_n)).\]
    Namely, $\varphi_K\circ a_K(T)=\varphi_K$ holds true.
\end{proof}

Let $(T_t)_{t\geq0}$ be a strongly continuous contraction semigroup on $\mathcal{H}$ commuting with $K$. By Equation \eqref{eq:qf_map}, we have
\begin{align*}
    a_K(T_t) a_K(T_s)(:a^{\#_1}(h_1)\cdots a^{\#_n}(h_n):_K)
     & =:a^{\#_1}(T_tT_sh_1)\cdots a^{\#_n}(T_tT_sh_n):_K   \\
     & =a_K(T_{t+s})(:a^{\#_1}(h_1)\cdots a^{\#_n}(h_n):_K)
\end{align*}
for every $h_1, \dots, h_n\in\mathcal{H}$ and $t, s\geq0$. Thus, we have $a_K(T_t) a_K(T_s)=a_K(T_{t+s})$ since the linear span of the $:a^{\#_1}(h_1)\cdots a^{\#_n}(h_n):_K$ ($h_1, \dots, h_n\in\mathcal{H}$) is dense in $\mathfrak{A}(\mathcal{H})$. By Equation \eqref{eq:qf_map_calc}, for any $h_1, \dots, h_m, k_1, \dots, k_n\in\mathcal{H}$ we have
\[\lim_{t\searrow 0}a_K(T_t)(a^*(h_m)\cdots a^*(h_1)a(k_1)\cdots a(k_n))=a^*(h_m)\cdots a^*(h_1)a(k_1)\cdots a(k_n)\]
in the norm topology. Thus, $\lim_{t\searrow0}a_K(T_t)(x)=x$ holds for all $x\in \mathfrak{A}(\mathcal{H})$. Namely, $(a_K(T_t))_{t\geq 0}$ forms a strongly continuous semigroup of (contractive) completely positive maps on $\mathfrak{A}(\mathcal{H})$.

In what follows, we use the same symbols $a_K(T_t)$ ($t\geq0$) to denote the restriction $a_K(T_t)|_{\mathfrak{I}(\mathcal{H})}$ to the GICAR algebra $\mathfrak{I}(\mathcal{H})$. Let $A$ denote the infinitesimal generator of $(T_t)_{t\geq0}$ and $da_K(A)$ the infinitesimal generator of $(a_K(T_t))_{t\geq 0}$ on $\mathfrak{I}(\mathcal{H})$. If $h_1, \dots, h_n$ belong to the domain of $A$, then, by Equation \eqref{eq:qf_map}, we have
\begin{align}\label{eq:qf_generator}
     & da_K(A):a^{\#_1}(h_1)\cdots a^{\#_n}(h_n):_K\nonumber                                                              \\
     & =\sum_{i=1}^n:a^{\#_1}(h_1)\cdots a^{\#_{i-1}}(h_{i-1})a^{\#_i}(Ah_i)a^{\#_{i+1}}(h_{i+1})\cdots a^{\#_n}(h_n):_K,
\end{align}
where $a^{\#_i}$ denotes either $a$ or $a^*$.

Since the quasi-free state $\varphi_K$ is invariant under $(a_K(T_t))_{t\geq0}$, we obtain the following result:
\begin{proposition}\label{prop:constract_semigroup}
    Let $(T_t)_{t\geq0}$ be a strongly continuous contraction semigroup on $\mathcal{H}$ such that $T_tK=KT_t$ ($t\geq0$). The following holds true:
    \begin{enumerate}
        \item There exists a unique strongly continuous semigroup $(\Lambda_K(T_t))_{t\geq0}$ on $\mathcal{H}_K$ such that $\Lambda_K(T_t)\rho_K(x)\Omega_K=\rho_K(a_K(T_t)(x))\Omega_K$ for every $x\in\mathfrak{I}(\mathcal{H})$ and $t\geq 0$.
        \item Let $A$ and $d\Lambda_K(A)$ denote the infinitesimal generators of $(T_t)_{t\geq0}$ and $(\Lambda_K(T_t))_{t\geq0}$. If $\mathcal{V}$ is a dense subspace in $\mathcal{H}$ such that $\mathcal{V}\subseteq \mathrm{dom}(A)$ and $T_t\mathcal{V}\subseteq \mathcal{V}$ for any $t\geq 0$, then
              \[\widetilde{\mathcal{V}}:=\mathrm{span}\{\rho_K(:a^*(h_1)\cdots a^*(h_n)a(k_1)\cdots a(k_n):_K)\Omega_K\mid h_1, \dots, h_n, k_1, \dots, k_n\in\mathcal{V}\}\]
              is a core of $d\Lambda_K(A)$.
    \end{enumerate}
\end{proposition}
\begin{proof}
    By Kadison's inequality (see e.g., \cite[Proposition II.6.9.14]{Blackadar:book}), we have
    \[\|\rho_K(a_K(T_t)(x))\Omega_K\|^2=\varphi_K(a_K(T_t)(x)^*a_K(T_t)(x))\leq \varphi_K(a_K(T_t)(x^*x))=\|\rho_K(x)\Omega_K\|^2\]
    for every $x\in\mathfrak{I}(\mathcal{H})$ and $t\geq 0$. Thus, the mapping $\rho_K(x)\Omega_K\mapsto \rho_K(a_K(T_t)(x))\Omega_K$ is well defined and contractive on $\rho_K(\mathfrak{I}(\mathcal{H}))\Omega_K$. Since $\rho_K(\mathfrak{I}(\mathcal{H}))\Omega_K$ is dense in $\mathcal{H}_K$, the above map can be uniquely extended to a contraction operator $\Lambda_K(T_t)$ on $\mathcal{H}_K$. Moreover, $(\Lambda_K(T_t))_{t\geq 0}$ forms a strongly continuous semigroup of contraction operators on $\mathcal{H}_K$ since $(a_K(T_t))_{t\geq0}$ is a strongly continuous semigroup.

    By Equations \eqref{eq:qf_map}, \eqref{eq:qf_generator}, $\widetilde{\mathcal{V}}$ is invariant under $(\Lambda_K(T_t))_{t\geq 0}$ and contained in $\mathrm{dom}(d\Lambda_K(A))$. Thus, by \cite[Themrem X.49]{RS:book2}, $\widetilde{\mathcal{V}}$ is a core for $d\Lambda_K(A)$.
\end{proof}

\subsection{Discrete orthogonal polynomial ensembles of hypergeometric type}
We return to the setting of Section \ref{sec:hyper_dop}. Namely, let $\mathfrak{X}=\mathbb{Z}+a$ or $\mathbb{Z}_\geq0+a$ for some $a\in\mathbb{R}$, and a weight function $w\colon\mathfrak{X}\to\mathbb{R}_{>0}$ has finite moments of all orders $n\geq0$. Moreover, we assume that there exist two polynomials $\sigma(x)$ and $\tau(x)$ such that $\deg\sigma(x)\geq2$, $\deg\tau(x)\leq1$, and
\[\Delta[\sigma(x)w(x)]=\tau(x)w(x), \quad \sigma(x)>0\]
holds on $\mathfrak{X}$. Thus, we obtain monic orthogonal polynomials $(\tilde p_n(x))_{n\geq0}$ satisfying Equation \eqref{eq:hypergeometric} with $\lambda=-m_n$ for all $n\geq0$, where $m_n:=\tau'n+\sigma''n(n-1)/2$. We also obtain the orthonormal basis $(p_n)_{n\geq0}$ for $\ell^2(\mathfrak{X})$ by $p_n(x):=\tilde p_n(x)w(x)^{1/2}/\|\tilde p_n\|_{L^2(\mathfrak{X}, w)}$. Then, they are eigenfunctions of $D$ defined by Equation \eqref{eq:differene_operator}. More precisely, we have $Dp_n=m_np_n$ for any $n\geq0$. We also assume that eigenvalues satisfy $m_0=0>m_1>m_2>\cdots$.

Let us fix $N\geq 0$ and $K$ denotes the orthogonal projection onto the linear span of $p_0, \dots, p_{N-1}$. For $\mu\in (m_N, m_{N-1})$ we define $B_\mu:=(1-K)(D+\mu)-K(D+\mu)$. By definition, $BK=KB$ holds, and $B$ is self-adjoint and negative semidefinite. Hence, by \cite[Corollary in Section X.8]{RS:book2}, we obtain a semigroup $(e^{tB})_{t\geq 0}$ of contraction operators on $\ell^2(\mathfrak{X})$ commuting with $K$. Therefore, by Proposition \ref{prop:constract_semigroup}, we obtain the semigroup $(\Lambda_K(e^{tB}))_{t\geq 0}$ of contraction operators on $\mathcal{H}_K$. Let $d\Lambda_K(B)$ denote its infinitesimal generator.

We recall that for any $\alpha_1> \cdots >\alpha_n\geq0$ and $N-1\geq \beta_1>\cdots >\beta_n\geq0$ we have
\begin{align*}
    p_{(\alpha|\beta)}
     & =\rho_K(a^*(p_{\alpha_1+N})\cdots a^*(p_{\alpha_n+N})a(p_{N-1-\beta_1})\cdots a(p_{N-1-\beta_n}))\Omega_K            \\
     & =(-1)^n\rho_K(:a^*(p_{\alpha_1+N})\cdots a^*(p_{\alpha_n+N})a(p_{N-1-\beta_1})\cdots a(p_{N-1-\beta_n}):_K)\Omega_K.
\end{align*}

Thus, by Proposition \ref{prop:constract_semigroup}(2), the linear span of the $p_{(\alpha|\beta)}$ is a core for $d\Lambda_K(B)$. Moreover, we have the following statement:
\begin{proposition}\label{prop:markov_semigroup}
    For any $t\geq 0$ and $p_{(\alpha|\beta)}$ we have $\Lambda_K(e^{tB})p_{(\alpha|\beta)}=e^{t(m_\lambda-m_\emptyset)}p_{(\alpha|\beta)}$, where $\lambda$ is a partition with Frobenius coordinate $(\alpha|\beta)$. In particular, $d\Lambda_K(B)p_{(\alpha|\beta)}=(m_\lambda-m_\emptyset)p_{(\alpha|\beta)}$.
\end{proposition}
\begin{proof}
    The latter assertion follows from the former. By Equation \eqref{eq:qf_map} and Lemma \ref{lem:eigenvalue}, we have
    \begin{align*}
         & \Lambda_K(e^{tB})p_{(\alpha|\beta)}                                                                                                         \\
         & =(-1)^n\rho_K(:a^*(e^{tB}p_{\alpha_1+N})\cdots a^*(e^{tB}p_{\alpha_n+N})a(e^{tB}p_{N-1-\beta_1})\cdots a(e^{tB}p_{N-1-\beta_n}):_K)\Omega_K \\
         & =\exp\left(t\sum_{i=1}^N((m_{\alpha_i+N}+\mu)-(m_{N-1-\beta_i}+\mu))\right)p_{(\alpha|\beta)}                                               \\
         & =\exp(t(m_\lambda-m_\emptyset))p_{(\alpha|\beta)}.
    \end{align*}
\end{proof}

Now we can compare two infinitesimal generators $d\Gamma_K(D)$ and $d\Lambda_K(B)$ of two dynamics $(\Gamma_K(e^{\mathrm{i}t D}))_{t\geq0}$ and $(\Lambda_K(e^{tB}))_{t\geq0}$.
\begin{corollary}\label{cor:love}
    The operator $d\Gamma_K(D)$ is equal to $d\Lambda_K(B)$.
\end{corollary}
\begin{proof}
    By Propositions \ref{prop:capsule} \ref{prop:markov_semigroup}, two operators $d\Gamma_K(D)$ and $d\Lambda_K(B)$ are the same on the linear span of the $p_{(\alpha|\beta)}$ for all $(\alpha|\beta)=(\alpha_1>\cdots >\alpha_n|\beta_1>\cdots >\beta_n)$. Moreover, this is cores for $d\Gamma_K(D)$ and $d\Lambda_K(B)$, and hence $d\Gamma_K(D)=d\Lambda_K(B)$.
\end{proof}

Let us recall that there exists an isomorphism $\mathcal{H}_K\cong L^2(\mathfrak{X}^{(N)}, M_{w, N})$ assigning $p_{(\alpha|\beta)}$ to $F_\lambda$ given in Definition \ref{def:second_isomorphism}, where $\lambda$ is a partition with Frobenius coordinate $(\alpha|\beta)$. Under the isomorphism, the semigroup $(\Lambda_K(e^{tB}))_{t\geq0}$ induces a semigroup of contraction operators on $L^2(\mathfrak{X}^{(N)}, M_{w, N})$. In what follows, it is denoted by the same symbol $(\Lambda_K(e^{tB}))_{t\geq0}$.

In the rest of the section, we determine kernels corresponding to $\Lambda_K(e^{tB})$. Let $P_t$ be a kernel on $\mathfrak{X}^{(N)}\times \mathfrak{X}^{(N)}$ such that
\[[\Lambda_K(e^{tB})f](\underline{x})=\sum_{\underline{y}\in \mathfrak{X}^{(N)}}P_t(\underline{x}, \underline{y})f(\underline{y})\]
for any $f\in L^2(\mathfrak{X}^{(N)}, M_{w, N})$. Since $(\Lambda_K(e^{tB}))_{t\geq0}$ is a semigroup, the kernels $(P_t)_{t\geq0}$ have a semigroup property, that is, for any $t, s\geq0$ we have
\[\sum_{\underline{y}\in\mathfrak{X}^{(N)}}P_t(\underline{x}, \underline{y})P_s(\underline{y}, \underline{z})=P_{t+s}(\underline{x}, \underline{z})\quad (\underline{x}, \underline{z}\in\mathfrak{X}^{(N)}).\]
We recall that $F_\emptyset=\mathbbm{1}$, where $\mathbbm{1}$ is the constant function equal to 1 (see Example \ref{ex:constant_one}). Thus, by Proposition \ref{prop:markov_semigroup}, we have $\Lambda_K(e^{tB})\mathbbm{1}=e^{t(m_\emptyset-m_\emptyset)}\mathbbm{1}=\mathbbm{1}$ for all $t\geq0$. Hence, for all $t\geq 0$ we have
\[\sum_{\underline{y}\in \mathfrak{X}^{(N)}}P_t(\underline{x}, \underline{y})=1\quad (\underline{x}\in \mathfrak{X}^{(N)}).\]

Moreover, we obtain the following explicit formula for the kernels $P_t$ ($t\geq0$).
\begin{theorem}\label{thm:formula_kernel}
    For any $\underline{x}, \underline{y}\in \mathfrak{X}^{(N)}$ and $t\geq 0$ we have
    \begin{equation}\label{eq:transition_prob}
        P_t(\underline{x}, \underline{y})=e^{-tm_\emptyset}\frac{V_N(\underline{y})}{V_N(\underline{x})}\det\left[ \frac{\langle e^{tD}\delta_{y_i}, \delta_{x_j}\rangle}{\langle \delta_{x_j}, \delta_{x_j}\rangle} \right]_{i, j=1}^N,
    \end{equation}
    where $V_N(\underline{x})=\prod_{1\leq i<j\leq N}(x_i-x_j)$ if $\underline{x}=\{x_1<\cdots <x_N\}$. In particular, $(P_t)_{t\geq0}$ forms a Markov semigroup if $D$ is an infinitesimal generator of birth death process.
\end{theorem}
\begin{proof}
    First, we prove Equation \eqref{eq:transition_prob}. By definition, $P_t(\underline{x}, \underline{y})=\langle \Lambda_K(e^{tB})\delta_{\underline{y}}, \delta_{\underline{x}}\rangle/\langle\delta_{\underline{x}}, \delta_{\underline{x}}\rangle$ holds, and hence
    \begin{align*}
        P_t(\underline{x}, \underline{y})
         & =M_{w, N}(\underline{x})^{-1}\langle P_t\delta_{\underline{y}}, \delta_{\underline{x}}\rangle                                                                                                                            \\
         & =M_{w, N}(\underline{x})^{-1}\sum_{\lambda=(\lambda_1\geq \cdots \geq \lambda_N)}\langle \delta_{\underline{y}}, F_\lambda \rangle \langle \Lambda_K(e^{tB})F_\lambda, \delta_{\underline{x}}\rangle                     \\
         & =\sum_{\lambda=(\lambda_1\geq \cdots \geq \lambda_N)}e^{t(m_\lambda-m_\emptyset)}F_\lambda(\underline{y})M_{w, N}(\underline{y}) F_\lambda(\underline{x})                                                                \\
         & =e^{-tm_\emptyset}\frac{V_N(\underline{y})}{V_N(\underline{x})}\prod_{j=1}^N\frac{w(y_j)^{1/2}}{w(x_j)^{1/2}}\sum_{\lambda}e^{tm_\lambda}\det[p_{\lambda_i+N-i}(x_j)]_{i, j=1}^N\det[p_{\lambda_i+N-i}(y_j)]_{i, j=1}^N.
    \end{align*}
    Moreover, we have
    \begin{align*}
         & \sum_{\lambda_1\geq \cdots \geq \lambda_N\geq 0}e^{tm_\lambda}\det[p_{\lambda_i+N-i}(x_j)]_{i, j=1}^N\det[p_{\lambda_i+N-i}(y_j)]_{i, j=1}^N                       \\
         & =\frac{1}{N!}\sum_{l_1, \cdots, l_N\geq0}e^{t\sum_{i=1}^Nm_{l_i}}\det[p_{l_i}(x_j)]_{i, j=1}^N\det[p_{l_i}(y_j)]_{i, j=1}^N                                        \\
         & =\det\left[\sum_{l=0}^\infty e^{tm_l}p_l(x_j)p_l(y_i)\right]_{i, j=1}^\infty                                                                                       \\
         & =\det\left[\frac{\langle e^{tD}\delta_{y_i}, \delta_{x_j}\rangle}{\langle \delta_{x_j}, \delta_{x_j}\rangle} \frac{w(x_i)^{1/2}}{w(y_j)^{1/2}} \right]_{i, j=1}^N.
    \end{align*}
    Namely, we obtain Equation \eqref{eq:transition_prob}. The second assertion follows from the Karlin--McGregor formula. Indeed, the determinant in Equation \eqref{eq:transition_prob} is positive. Moreover, it is equal to the probability that $N$ independent birth death processes generated by $D$ start from $x_1<\cdots <x_N$ and reach $y_1<\cdots <y_N$ at time $t$ without coincident of any two processes of them. See \cite{KM57,KM59}. Hence, $(P_t)_{t\geq 0}$ forms a Markov semigroup.
\end{proof}

For instance, it is known that the Meixner and Charlier difference operators (in Examples \ref{ex:Meixner}, \ref{ex:Charlier}) are infinitesimal generators of birth deal process on $\mathbb{Z}_{\geq0}$ (see e.g., \cite[Corollary 6.6]{BO13}). Thus, by Theorem \ref{thm:formula_kernel}, the corresponding dynamics $(\Lambda_K(e^{tB})_{t\geq0})$ give Markov semigroups on $\mathbb{Z}_{\geq0}$. Moreover, the corresponding orthogonal polynomial ensembles are invariant probability measure of these Markov semigroups. Let $\mathbb{Y}_{\leq N}$ denote the set of partitions with length $\leq N$. We remark that there exists a bijection $\mathbb{Y}_{\leq N}\cong \mathbb{Z}_{\geq 0}^{(N)}$ ($\cong \mathcal{C}_N(\mathbb{Z}_{\geq 0})$) given as $\lambda\leftrightarrow (\lambda_i+N-i)_{i=1}^N$. Thus, Markov semigroups on $\mathbb{Z}_{\geq 0}^{(N)}$ induce Markov semigroups on partitions, and such stochastic dynamics related to Meixner and Charlier polynomials are investigated in \cite{BO06,Olshanski11,Olshanski12, BO13,Petrov13}, etc. In the paper, we have established an operator algebraic understanding of such stochastic dynamics.

\subsection{The hypergeometric difference operator}\label{subsect:stoc_hypergeometric}
Here we discuss a stochastic dynamics generated by the hypergeometric difference operator $D_{z, z', \xi}$ in Example \ref{ex:z-measure}, where the pair $z, z'\in\mathbb{C}$ is principal or complementary and $\xi\in (0, 1)$. Let us recall that $\mathbb{Z}'=\mathbb{Z}+1/2$ and there exists an orthonormal basis $(\psi_a)_{a\in\mathbb{Z}'}$ such that $D_{z, z', \xi}\psi_a=(1-\xi)a\psi_a$ for every $a\in\mathbb{Z}'$. The kernel $K_{z, z', \xi}$ is defined as the same as Section \ref{sec:hgdo}, and it is a correlation kernel of the $z$-measure $M_{z, z', \xi}$. As mentioned in Section \ref{sec:hgdo}, the integral operator $K$ corresponding to $K_{z, z', \xi}$ is the orthogonal projection onto $\overline{\mathrm{span}}\{\psi_a\mid a\in\mathbb{Z}'_+\}$. Following the previous discussion, we denote by $\varphi_K$ and $(\rho_K, \mathcal{H}_K, \Omega_K)$ the corresponding quasi-free state of the GICAR algebra $\mathfrak{I}(\mathbb{Z}'):=\mathfrak{I}(\ell^2(\mathbb{Z}'))$ and the associated GNS-triple, respectively.

In Section \ref{sec:hgdo}, we gave an isomorphism $\mathcal{H}_K\cong L^2(\mathbb{Y}, M_{z, z', \xi})$ by  fixing an orthonormal basis for $L^2(\mathbb{Y}, M_{z, z', \xi})$ labeled by partitions $\mathbb{Y}$. Here we fix such an orthonormal basis concretely. First, for any $\mu\in \mathbb{Y}$ we define a function $FS_\mu$ on $\mathbb{Y}$ by
\[FS_\mu(\lambda):=\begin{cases}|\lambda|^{\downarrow |\mu|}\frac{\dim \lambda/\mu}{\dim \lambda}& \text{if }\mu\subseteq\lambda, \\ 0&\text{otherwise},\end{cases}\]
where $l^{\downarrow m}:=l(l-1)\cdots (l-m+1)$ and $\dim \lambda/\mu$ is the number of standard tableaux of shape $\lambda/\mu$. We remark that $\dim \lambda/\mu$ is well defined only if $\mu\subseteq\lambda$, that is, $\mu_i\leq \lambda_i$ for all $i=1, 2, \dots$. Moreover, we define a function $\mathfrak{M}_\lambda$ ($\lambda\in\mathbb{Y}$) by
\[\mathfrak{M}_\lambda:=\sum_{\mu\subseteq \lambda}\left(\frac{-\xi}{1-\xi}\right)^{|\lambda|-|\mu|}\frac{\dim\lambda/\mu}{(|\lambda|-|\mu|)!}(z)_{\lambda/\mu}(z')_{\lambda/\mu}FS_\mu,\]
where $(z)_{\lambda/\mu}:=\prod_{i=1}^l\prod_{j=\mu_i+1}^{\lambda_i}(z+j-i)$ when $l$ is the length of $\lambda$. By definition, if $\mu=\emptyset$, then $(z)_{\lambda/\emptyset}=(z)_\lambda$ holds. By \cite[Theorem 5.27(ii)]{Olshanski12}, these functions $(\mathfrak{M}_\lambda)_{\lambda\in\mathbb{Y}}$ form a \emph{orthogonal} basis for $L^2(\mathbb{Y}, M_{z, z', \xi})$. Using this basis, we define an isomorphism $\mathcal{H}_K\cong L^2(\mathbb{Y}, M_{z, z',\xi})$ assigning $\psi_{\alpha_1, \dots, \alpha_n; \beta_1, \dots, \beta_n}$ to $\mathfrak{M}_\lambda/\|\mathfrak{M}_\lambda\|_{L^2(\mathbb{Y}, M_{z, z', \xi})}$ for any $\alpha_1>\cdots >\alpha_n\geq 1/2$ and $\beta_1>\cdots >\beta_n\geq 1/2$, where $\lambda$ is a partition with modified Frobenius coordinate $(\alpha_1, \dots, \alpha| \beta_1, \dots, \beta_n)$.

Similar to the previous case, for any $\alpha_1>\cdots >\alpha_n\geq 1/2$ and $\beta_1>\cdots >\beta_n\geq1/2$ we have
\begin{align*}
    \psi_{\alpha_1, \dots, \alpha_n; \beta_1, \dots, \beta_n}
     & =\rho_K(a^*(\psi_{-\alpha_1})\cdots a^*(\psi_{-\alpha_n})a(\psi_{\beta_1})\cdots a(\psi_{\beta_n}))\Omega_K            \\
     & =(-1)^n\rho_K(:a^*(\psi_{-\alpha_1})\cdots a^*(\psi_{-\alpha_n})a(\psi_{\beta_1})\cdots a(\psi_{\beta_n}):_K)\Omega_K.
\end{align*}
Thus, by Proposition \ref{prop:constract_semigroup}(2), the linear span of them is a core for $d\Lambda_K(B)$.

Let $B_{z, z', \xi}:=(1-\xi)^{-1}(D_{z, z', \xi}(1-K)-D_{z, z', \xi}K)$. Namely, for any $a\in\mathbb{Z}'$ we have
\[B_{z, z', \xi}\psi_a=\begin{cases}-a\psi_a & (a\geq 1/2), \\ a\psi_a &(a \leq 1/2),\end{cases}\]
and hence $B_{z, z', \xi}$ is self-adjoint and negative semidefinite. Thus, \cite[Corollary in Secion X.8]{RS:book2}, $B_{z, z', \xi}$ generates a strongly continuous semigroup $(e^{tB_{z, z', \xi}})_{t\geq0}$ of contractions on $\ell^2(\mathbb{Z}')$, and it commutes with $K$. By Proposition \ref{prop:constract_semigroup}(1) and the above isomorphism $\mathcal{H}_K\cong L^2(\mathbb{Y}, M_{z, z',\xi})$, we obtain a semigroup $(\Lambda_K(e^{tB}))_{t\geq0}$ of contraction operators on $L^2(\mathbb{Y}, M_{z, z', \xi})$. Let $d\Lambda_K(B_{z, z', \xi})$ denote its infinitesimal generator. Similar to Proposition \ref{prop:markov_semigroup}, the following holds true:
\begin{proposition}\label{prop:will}
    For any $t\geq 0$ and $\lambda\in\mathbb{Y}$ we have
    \[\Lambda_K(e^{tB_{z, z', \xi}})\mathfrak{M}_\lambda=e^{-t|\lambda|}\mathfrak{M}_\lambda, \quad d\Lambda_K(B_{z, z', \xi})\mathfrak{M}_\lambda=-|\lambda|\mathfrak{M}_\lambda.\]
\end{proposition}

The following is a consequence of Propositions \ref{prop:ocha}, \ref{prop:will} (see also Corollary \ref{cor:love})
\begin{corollary}
    The operator $d\Gamma_K(D_{z, z', \xi})$ is equal to $d\Lambda_K(B_{z, z', \xi})$.
\end{corollary}

By results in \cite{BO13}, the semigroup $(\Lambda_K(e^{t B_{z, z', \xi}}))_{t \geq 0}$ has a stochastic interpretation, and it is a reason why we chose a basis $(\mathfrak{M}_\lambda)_{\lambda\in \mathbb{Y}}$ for $L^2(\mathbb{Y}, M_{z, z', \xi})$. To explain it, we use the following notation. For two Young diagrams $\lambda, \nu\in\mathbb{Y}$, we write $\lambda\nearrow \nu$ if $\nu$ is obtained by adding one box to $\lambda$. Similarly, we write $\nu\searrow \lambda$ if $\nu$ is obtained by deleting one box from $\lambda$. We define a kernel $Q_{z, z', \xi}$ on $\mathbb{Y}\times \mathbb{Y}$ as follows (see \cite[Equation (8.11)]{BO13}):
\begin{align*}
    (1-\xi)Q_{z, z', \xi}(\lambda, \nu)
    :=\begin{cases}
          \xi (z)_{\nu/\lambda}(z')_{\nu/\lambda}\frac{\dim \nu}{(|\lambda|+1)\dim \lambda} & \lambda\nearrow \nu,  \\
          \frac{|\lambda|\dim \nu}{\dim \nu}                                                & \lambda \searrow \nu, \\
          (1+\xi)|\lambda| + \xi z z'                                                       & \nu = \lambda,        \\
          0                                                                                 & \text{otherwise}.
      \end{cases}
\end{align*}
By \cite[Proposition 8.4]{BO13}, \cite[Proposition 4.24]{Olshanski12}, we have $Q_{z, z', \xi}\mathfrak{M}_\lambda=-|\lambda|\mathfrak{M}_\lambda$. Namely, $Q_{z, z', \xi}$ is equal to $d\Lambda(B_{z, z', \xi})$ on the linear span of the $\mathfrak{M}_\lambda$ ($\lambda\in\mathbb{Y}$). Moreover, by \cite[Corollary 8.7]{BO13}, $Q_{z, z', \xi}$ gives rise to a Markov (Feller) semigroup on $\mathbb{Y}$. In other words, $(\Lambda_K(e^{t B_{z, z', \xi}}))_{t\geq 0}$ preserves the space $C_0(\mathbb{Y})$ of (continuous) functions vanishing at infinity, and it is a Feller semigroup, that is, a strongly continuous, positive, conservative contraction semigroup.

\section{Conclusions}
As we observed in Section \ref{sec:DPPGICAR}, determinantal point processes on discrete spaces are intimately related to quasi-free states on GICAR algebras. In the paper, we extended this relation to encompass dynamical aspects. In particular, we formulated a connection between dynamics on determinantal point processes and dynamics on GICAR algebras. The essential ideas can be summarized as follows: Firstly, for a specific determinantal point process, its correlation kernel is provided from a quasi-free state on a GICAR algebra, and the associated GNS-representation space is isomorphic to the $L^2$-space of the determinantal point process. Secondly, dynamics on a GICAR algebra preserving a quasi-free state induce dynamics on the associated GNS-representation space and the $L^2$-space with respect to the determinantal point process. In Section \ref{sec:unitary_dynamics}, we delve into semigroups of Bogoliubov automorphisms on a GICAR algebra and show that it gives rise to dynamics of unitary operators on the $L^2$-space of a determinantal point process. Furthermore, in Section \ref{sec:stochastic_dynamics}, we explore semigroups of completely positive linear maps on a GICAR algebra and show that it gives rise to stochastic dynamics on determinantal point processes. As Koshida \cite{Koshida21} pointed out, Pfaffian point processes on discrete spaces also relate to quasi-free states on self-dual GICAR algebras. Hence, by the method in the paper, we can establish a framework for constructing dynamics on Pfaffian point processes from dynamics on self-dual GICAR algebras. This will serve as a subject of future research.

In the paper, we dealt with two types of examples of determinantal point processes, which arise from orthogonal polynomials of hypergeometric type and $z$-measures on the set $\mathbb{Y}$ of Young diagrams. Those correlation kernels are given as spectral projections of self-adjoint difference operators, and they produce appropriate dynamics on GICAR algebras and unitary or stochastic dynamics on these determinantal point processes. Moreover, there are other determinantal point processes whose correlation kernels are given as spectral projections of self-adjoint operators. For instance, we did not discuss determinantal point processes related to $q$-orthogonal polynomials in the paper. However, they are interesting subjects in the intersection of probability theory and representation theory, see, e.g., \cite{GO16}.

Determinantal point processes on continuum spaces are also quite interesting subjects and intimately related to operator algebras (see \cite{Lytvynov02,LM07}). Hence, it is only natural to inquire about extending our results to the continuum case. Particularly, when exploring determinantal point processes with a representation-theoretic origin, we frequently delve into the continuum limits of discrete determinantal point processes and their dynamical aspects (see \cite{BO05, BO06, Olshanski08, BO13, GO16}). Therefore, it is an intriguing problem to reveal how such continuum limits, encompassing dynamical aspects, are described within the framework of operator algebras.

\section*{Acknowledgement}
The author gratefully acknowledges comments from Professor Makoto Katori since the early stage of this work. The author is grateful to Professor Gregory Schehr, Professor Tomoyuki Shirai, Professor Takeshi Imamura and Mr. Shinji Koshida for the valuable discussions. This work was supported by JSPS Research Fellowship for Young Scientists PD (KAKENHI Grand Number 22J00573).
}

\end{document}